\documentclass[12pt,letterpaper,final,twoside,leqno]{amsart}
\usepackage{eucal, mathrsfs}
\usepackage{amsmath,amssymb,amsthm,%amsfonts,
amscd,amsopn}
\usepackage{hyperref}
\usepackage{times}
\usepackage{xspace}
\usepackage{epsfig,epic,eepic,latexsym,color}
\usepackage[all]{xy}
\usepackage{paralist}
\usepackage{stmaryrd}
\usepackage{enumitem}
\usepackage{color}
\usepackage{soul}
\usepackage{mathtools}

\catcode`~=11 \def\UrlSpecials{\do\~{\kern -.15em\lower .7ex\hbox{~}\kern .04em}} \catcode`~=13

\newcommand{\urlwofont}[1]{\urlstyle{same}\url{#1}}

\newcommand{\widepagestyle}{
\voffset=0in
\hoffset=0in
\marginparwidth=0.7in
\oddsidemargin=0in
\evensidemargin=0in
\textwidth=6.5in
\textheight=8.5in
\topmargin=0in
\headheight=0in
\headsep=0.2in
\footskip=0.5in
}

\newcounter{are-there-sections}
\makeatletter

\renewcommand\subsection{
  \renewcommand{\sfdefault}{pag}
  \@startsection{subsection}%
  {2}{0pt}{-\baselineskip}{.2\baselineskip}{\raggedright
    \sffamily\itshape\small\bfseries
  }}
\renewcommand\section{
  \renewcommand{\sfdefault}{phv}
  \@startsection{section} %
  {1}{0pt}{\baselineskip}{.2\baselineskip}{\centering
    \sffamily
    \scshape
    \bfseries
}}
\newcounter{lastyear}
\addtocounter{lastyear}{-1}
\newcommand\noin{\noindent}

\newtheoremstyle{bozont}{8pt}{10pt}%
     {\itshape}%         Body font
     {}%         Indent amount (empty = no indent, \parindent = para indent)
     {\bfseries}% Thm head font
     {.}%        Punctuation after thm head
     {.5em}%     Space after thm head (\newline = linebreak)
     {\thmname{#1}\thmnumber{ #2}\thmnote{ \rm #3}}%         Thm head spec
\newtheoremstyle{bozont-sf}{3pt}{3pt}%
     {\itshape}%         Body font
     {}%         Indent amount (empty = no indent, \parindent = para indent)
     {\sffamily}% Thm head font
     {.}%        Punctuation after thm head
     {.5em}%     Space after thm head (\newline = linebreak)
     {\thmname{#1}\thmnumber{ #2}\thmnote{ \rm #3}}%         Thm head spec
\newtheoremstyle{bozont-sc}{3pt}{3pt}%
     {\itshape}%         Body font
     {}%         Indent amount (empty = no indent, \parindent = para indent)
     {\scshape}% Thm head font
     {.}%        Punctuation after thm head
     {.5em}%     Space after thm head (\newline = linebreak)
     {\thmname{#1}\thmnumber{ #2}\thmnote{ \rm #3}}%         Thm head spec
\newtheoremstyle{bozont-remark}{8pt}{15pt}%
     {}%         Body font
     {}%         Indent amount (empty = no indent, \parindent = para indent)
     {\scshape}% Thm head font
     {}%        Punctuation after thm head
     {.5em}%     Space after thm head (\newline = linebreak)
     {\thmname{#1}\thmnumber{ #2.}\thmnote{\ \  ( #3)}}%         Thm head spec
\newtheoremstyle{bozont-def}{8pt}{15pt}%
     {}%         Body font
     {}%         Indent amount (empty = no indent, \parindent = para indent)
     {\bfseries}% Thm head font
     {.}%        Punctuation after thm head
     {.5em}%     Space after thm head (\newline = linebreak)
     {\thmname{#1}\thmnumber{ #2}\thmnote{ \rm #3}}%         Thm head spec
\newtheoremstyle{bozont-reverse}{3pt}{3pt}%
     {\itshape}%         Body font
     {}%         Indent amount (empty = no indent, \parindent = para indent)
     {\bfseries}% Thm head font
     {.}%        Punctuation after thm head
     {.5em}%     Space after thm head (\newline = linebreak)
     {\thmnumber{#2.}\thmname{ #1}\thmnote{ \rm #3}}%         Thm head spec
\newtheoremstyle{bozont-reverse-sc}{3pt}{3pt}%
     {\itshape}%         Body font
     {}%         Indent amount (empty = no indent, \parindent = para indent)
     {\scshape}% Thm head font
     {.}%        Punctuation after thm head
     {.5em}%     Space after thm head (\newline = linebreak)
     {\thmnumber{#2.}\thmname{ #1}\thmnote{ \rm #3}}%         Thm head spec
\newtheoremstyle{bozont-reverse-sf}{3pt}{3pt}%
     {\itshape}%         Body font
     {}%         Indent amount (empty = no indent, \parindent = para indent)
     {\sffamily}% Thm head font
     {.}%        Punctuation after thm head
     {.5em}%     Space after thm head (\newline = linebreak)
     {\thmnumber{#2.}\thmname{ #1}\thmnote{ \rm #3}}%         Thm head spec
\newtheoremstyle{bozont-remark-reverse}{3pt}{3pt}%
     {}%         Body font
     {}%         Indent amount (empty = no indent, \parindent = para indent)
     {\sc}% Thm head font
     {.}%        Punctuation after thm head
     {.5em}%     Space after thm head (\newline = linebreak)
     {\thmnumber{#2.}\thmname{ #1}\thmnote{ \rm #3}}%         Thm head spec
\newtheoremstyle{bozont-def-reverse}{3pt}{3pt}%
     {}%         Body font
     {}%         Indent amount (empty = no indent, \parindent = para indent)
     {\bfseries}% Thm head font
     {.}%        Punctuation after thm head
     {.5em}%     Space after thm head (\newline = linebreak)
     {\thmnumber{#2.}\thmname{ #1}\thmnote{ \rm #3}}%         Thm head spec
\newtheoremstyle{bozont-def-newnum-reverse}{3pt}{3pt}%
     {}%         Body font
     {}%         Indent amount (empty = no indent, \parindent = para indent)
     {\bfseries}% Thm head font
     {}%        Punctuation after thm head
     {.5em}%     Space after thm head (\newline = linebreak)
     {\thmnumber{#2.}\thmname{ #1}\thmnote{ \rm #3}}%         Thm head spec
\theoremstyle{bozont}    
\ifnum \value{are-there-sections}=0 {%
  \newtheorem{proclaim}{Theorem}
} 
\else {%
  \newtheorem{proclaim}{Theorem}[section]
} 
\fi
\newtheorem{thm}[proclaim]{Theorem}

\newtheorem{cor}[proclaim]{Corollary} 
 
\newtheorem{lem}[proclaim]{Lemma} 
\newtheorem{prop}[proclaim]{Proposition} 
\newtheorem{conj}[proclaim]{Conjecture}

\theoremstyle{bozont-sc}
\newtheorem{proclaim-special}[proclaim]{\specialthmname}

\theoremstyle{bozont-remark}

\newtheorem{rem}[proclaim]{Remark}

\newtheorem*{SubHeading*}{\SubHeadingName}%
\newtheorem{SubHeading}[proclaim]{\SubHeadingName}
\newtheorem{sSubHeading}[equation]{\sSubHeadingName}
\newenvironment{demo-r}[1]{\def\SubHeadingName{#1}\begin{SubHeading-r}}
  {\end{SubHeading-r}}%
\newenvironment{subdemo-r}[1]{\def\sSubHeadingName{#1}\begin{sSubHeading-r}}
  {\end{sSubHeading-r}} %
\newenvironment{demo*}[1]{\def\SubHeadingName{#1}\begin{SubHeading*}}
  {\end{SubHeading*}}%

\newtheorem{question}[proclaim]{Question}

\newtheorem{defn-thm}[proclaim]{Definition--Theorem}  %!!!!!!!!!!!!!!!!!!!!!!!!
\theoremstyle{bozont-def}    
\newtheorem{defn}[proclaim]{Definition}
\newtheorem{notation}[proclaim]{Notation} 

\theoremstyle{bozont-reverse}    
\theoremstyle{bozont-reverse-sc}
\newtheorem{proclaimr-special}[proclaim]{\specialthmname}
{\def\specialthmname{#1}\begin{proclaimr-special}}%
{\end{proclaimr-special}}
\theoremstyle{bozont-remark-reverse}

\newtheorem{SubHeading-r}[proclaim]{\SubHeadingName}
\newtheorem{sSubHeading-r}[equation]{\sSubHeadingName}
\newtheorem{SubHeadingr}[proclaim]{\SubHeadingName}

\theoremstyle{bozont-def-newnum-reverse}    

\theoremstyle{bozont-def-reverse}

\newtheorem{newnumspecial}[proclaim]{\specialnewnumname}

\numberwithin{equation}{proclaim}

\numberwithin{figure}{section}

\newenvironment{enumerate-p}{
  \begin{enumerate}}
  {\setcounter{equation}{\value{enumi}}\end{enumerate}}
\newenvironment{enumerate-cont}{
  \begin{enumerate}
    {\setcounter{enumi}{\value{equation}}}}
  {\setcounter{equation}{\value{enumi}}
  \end{enumerate}}
\newlength{\swidth}
\makeatother
\DeclareMathAlphabet{\smallchanc}{OT1}{pzc}%
                                 {m}{it}
\DeclareFontFamily{OT1}{pzc}{}
\DeclareFontShape{OT1}{pzc}{m}{it}%
             {<-> s * [1.100] pzcmi7t}{}
\DeclareMathAlphabet{\mathchanc}{OT1}{pzc}%
                                 {m}{it}

\DeclareFontFamily{OMS}{rsfs}{\skewchar\font'60}
\DeclareFontShape{OMS}{rsfs}{m}{n}{<-5>rsfs5 <5-7>rsfs7 <7->rsfs10 }{}
\DeclareSymbolFont{rsfs}{OMS}{rsfs}{m}{n}
\DeclareSymbolFontAlphabet{\scr}{rsfs}

\newcommand{\sA}{\scr{A}}
\newcommand{\sB}{\scr{B}}
\newcommand{\sC}{\scr{C}}
\newcommand{\sD}{\scr{D}}
\newcommand{\sE}{\scr{E}}
\newcommand{\sF}{\scr{F}}
\newcommand{\sG}{\scr{G}}

\newcommand{\sI}{\scr{I}}
\newcommand{\sJ}{\scr{J}}
\newcommand{\sK}{\scr{K}}
\newcommand{\sL}{\scr{L}}
\newcommand{\sM}{\scr{M}}
\newcommand{\sN}{\scr{N}}
\newcommand{\sO}{\scr{O}}

\newcommand{\bP}{\mathbb{P}}
\newcommand{\bQ}{\mathbb{Q}}
\newcommand{\bR}{\mathbb{R}}

\newcommand{\bZ}{\mathbb{Z}}

\newcommand{\fM}{\mathfrak{M}}

\DeclareMathOperator{\ind}{{ind}}

\DeclareMathOperator{\Aut}{Aut}

\DeclareMathOperator{\codim}{codim}

\DeclareMathOperator{\Ker}{{Ker}}

\DeclareMathOperator{\id}{{id}}

\DeclareMathOperator{\im}{{im}}

\DeclareMathOperator{\ord}{{ord}}

\DeclareMathOperator{\Pic}{Pic}
\DeclareMathOperator{\Weil}{Weil}

\DeclareMathOperator{\sing}{sing}

\DeclareMathOperator{\Spec}{{Spec}}

\DeclareMathOperator{\Supp}{{Supp}}

\newcommand{\factor}[2]{\left. \raise 2pt\hbox{\ensuremath{#1}} \right/
        \hskip -2pt\raise -2pt\hbox{\ensuremath{#2}}}
\def\coh#1.#2.#3.{H^{#1}(#2,#3)}
\def\dimcoh#1.#2.#3.{h^{#1}(#2,#3)}
\def\hypcoh#1.#2.#3.{\mathbb H_{\vphantom{l}}^{#1}(#2,#3)}
\def\loccoh#1.#2.#3.#4.{H^{#1}_{#2}(#3,#4)}
\def\dimloccoh#1.#2.#3.#4.{h^{#1}_{#2}(#3,#4)}
\def\lochypcoh#1.#2.#3.#4.{\mathbb H^{#1}_{#2}(#3,#4)}
\def\ses#1.#2.#3.{0  \longrightarrow  #1   \longrightarrow 
 #2 \longrightarrow #3 \longrightarrow 0} 
\def\sesshort#1.#2.#3.{0
 \rightarrow #1 \rightarrow #2 \rightarrow #3 \rightarrow 0}
\def\dist#1.#2.#3.{  #1   \longrightarrow 
 #2 \longrightarrow #3 \stackrel{+1}{\longrightarrow} } % \tag{$\bigtriangleup$}}
\def\CDdist#1.#2.#3.{  #1   @>>>  #2  @>>>   #3 @>+1>> }  
\def\shortses#1.#2.#3.{0  \rightarrow  #1   \rightarrow 
 #2  \rightarrow   #3 \rightarrow  0}
\def\shortdist#1.#2.#3.{  #1   \rightarrow 
 #2  \rightarrow   #3 \stackrel{+1}{\rightarrow} }  % \tag{$\bigtriangleup$}}
\def\ddist#1.#2.#3.#4.#5.#6.{\CD
#1 @>>> #2 @>>> #3 @>+1>> \\
@VVV @VVV @VVV \\
#4 @>>> #5 @>>> #6 @>+1>> 
\endCD}
\def\ddistun#1.#2.#3.#4.#5.#6.{\CD
#1 @>>> #2 @>>> #3 @>+1>> \\
@. @VVV @VVV  \\
#4 @>>> #5 @>>> #6 @>+1>> 
\endCD}
\def\Iff#1#2#3{
\hfil\hbox{\hsize =#1
\vtop{\noin #2}
\hskip.5cm 
\lower.5\baselineskip\hbox{$\Leftrightarrow$}\hskip.5cm
\vtop{\noin #3}}\hfil\medskip}
\newcommand{\union}\cup
\newcommand{\intersect}\cap
\newcommand{\Union}\bigcup
\newcommand{\Intersect}\bigcap
\def\myoplus#1.#2.{\underset #1 \to {\overset #2 \to \oplus}}

\newcommand{\Quot}{\mathfrak{Quot}}

\widepagestyle

\address{Zsolt Patakfalvi, Princeton University, Department of Mathematics, Fine Hall, Washington Road,
Princeton, NJ-08544-1000, USA
}
\email{pzs@math.princeton.edu}
\urladdr{http://www.math.princeton.edu/\~pzs}

\theoremstyle{bozont}    
\newtheorem*{corollary_projectivity}{Corollary \ref{cor:projectivity_of_moduli}}
\newtheorem*{corollary_characteristic_zero}{Corollary \ref{cor:characteristic_zero}}
\newtheorem*{corollary_subadditivity}{Corollary \ref{cor:subadditivity_general_type_base}}

\title{Semi-positivity in positive characteristics}
\author{Zsolt Patakfalvi}

\begin{document}

\maketitle

\begin{abstract}
Let $f : (X, \Delta) \to Y$ be a flat, projective family of sharply $F$-pure, log-canonically polarized pairs over an algebraically closed field of characteristic $p >0$  such that $p \nmid \ind(K_{X/Y} + \Delta)$. We show that $K_{X/Y} + \Delta$ is nef and that $f_* (\sO_X(m (K_{X/Y} + \Delta)))$ is a nef vector bundle for $m \gg 0$ and divisible enough. Some of the results also extend to non log-canonically polarized pairs. The main motivation of the above results is projectivity of  proper subspaces of the moduli space of stable pairs in positive characteristics. Other applications are Kodaira vanishing free, algebraic proofs of corresponding positivity results in characteristic zero, and special cases of subadditivity of Kodaira-dimension in positive characteristics.
\end{abstract}

\tableofcontents

\section{Introduction}

Results stating positivity of the (log-)relative canonical bundle and of the pushforwards of its powers (e.g. \cite{Cartan_Plongements_projectifs_Seminaire_Henri_Cartan_Fonctions_automorphes}, \cite{Griffiths_Periods_of_integrals}, \cite{Fujita_On_Kahler_fiber_spaces}, \cite{Kawamata_Characterization_of_abelian_varieties}, \cite{Viehweg_Weak_positivity}, \cite{Kollar_Subadditivity_of_the_Kodaira_dimension}) played  an important role in the development of modern algebraic geometry. Applications are numerous: projectivity and quasi-projectivity of moduli spaces (e.g. \cite{Kollar_Projectivity_of_complete_moduli}, \cite{Viehweg_Quasi_projective_moduli}), subadditivity of Kodaira-dimension (e.g., \cite{Viehweg_Weak_positivity}, \cite{Kollar_Subadditivity_of_the_Kodaira_dimension}), Shafarevich type results about hyperbolicity of moduli spaces (e.g., \cite{Parshin_Algebraic_curves_over_function_fields}, \cite{Arakelov_Families_of_algebraic_curves}, \cite{Szpiro_Sur_le_theoreme_de_rigidite}), 
Kodaira dimension of moduli spaces (e.g., \cite{Mumford_Hirzebruch_s_proportionality_theorem}, \cite{Eisenbud_Harris_The_Kodaira_dimension_of_the_moduli_space_of_curves}), etc. Most of the proofs of the above mentioned general positivity results are either  analytic or  depend on Kodaira vanishing. Either way, they work only in characteristic zero. The word ``general'' and ``most'' has to be stressed here: there are positivity results available for families of curves (e.g., \cite{Szpiro_Sur_le_theoreme_de_rigidite}, \cite{Kollar_Projectivity_of_complete_moduli}) and K3 surfaces \cite{Maulik_Supersingular_K3_surfaces_for_large_primes} in positive characteristics. The aim of this article is 
to present positivity results available for arbitrary fiber dimensions in positive characteristics, bypassing the earlier used analytic or Kodaira vanishing type techniques. The strongest statements are in the case of (log)-canonically polarized fibers, but there are results for fibers with with nef log-canonical bundles as well. As in characteristic zero, one also has to put some restrictions on singularities. Here we assume the fibers to be sharply $F$-pure, which corresponds to characteristic zero notion of log-canonical singularities via reduction mod $p$ (see \cite{Schwede_Tucker_A_survey_of_test_ideals} for a survey on $F$-singularities, and Definition \ref{defn:non_F_pure} for the defintion of sharply $F$-pure singularities).  

Some differences between our results and the characteristic zero statements mentioned above have to be stressed. First, we only claim the semi-positivity of $f_* \omega_{X/Y}^m$ for $m$ big and divisible enough. This is a notable difference, since the characteristic zero results usually start with proving the $m=1$ case and then deduce the rest from that. However, in positive characteristics there are known counterexamples for the semi-positivity of $f_* \omega_{X/Y}$ \cite[3.2]{Moret_Bailly_Familles_de_courbes_et_de_varietes_abeliennes_sur_P_1_II_exemples}. So, any positivity result can hold only for $m >1$, and its proof has to bypass the $m=1$ case. Second, the characteristic zero results are birational in the sense that for example it is enough to assume that  $\omega_F$ is big for a general fiber of $F$. In our results nefness of $\omega_F$ is essential, and for the semi-positivity of pushforwards we even need $\omega_F$ to be ample. Hence, our results give exactly what one needs for projectivity of 
moduli 
spaces (as in \cite{Kollar_Projectivity_of_complete_moduli}), 
but yield subadditivity of Kodaira dimimension  only together with log-Minimal Model Program in positive characteristics.

% These assumptions are more suitable for the moduli setup, that is, for the 
% endproduct 
% of a  (log-)MMP relative to the base. With other words, our results yield birational results conditional on MMP in positive characteristics. 

\subsection{Results: normal, boundary free versions over a curve base}
\label{sec:results_initial_form}

Here we state our results in a special, but less technical form. We assume that the spaces involved are normal and we do not add  boundary divisors to our varieties. The base is also assumed to be a smooth projective curve. For the general form of the results see Section \ref{sec:results_full_generality}. 

We work over an algebraically closed field $k$ of characteristic $p>0$. 

\begin{thm}
\label{thm:results_initial_form}
Let $f : X \to Y$ be surjective, projective morphism from a normal variety  to a smooth projective curve with normal, sharply $F$-pure generic fiber. %Further assume that $K_X$ is $\bQ$-Cartier with index $r$ not divisible by $p$. 
Further assume that $rK_X$ is cartier for some $(p, r)=1$.
\begin{enumerate}
\item \label{itm:results_initial_form:nef} If $K_{X/Y}$ is $f$-nef and $K_{X_y}$ is semi-ample for generic $y \in Y$, then $K_{X/Y}$ is nef.
\item \label{itm:results_initial_form:push_forward} If $K_{X/Y}$ is $f$-ample, then $f_* \sO_X(mrK_{X/Y})$ is a nef vector bundle for $m \gg 0$.
\item (A subadditivity of Kodaira dimension type corollary:) If $K_{X/Y}$ is $f$-nef, $K_{X_y}$ is big for generic $y \in Y$ and $g(Y) \geq 2$, then $K_X$ is big as well.
\end{enumerate}
\end{thm}

\begin{rem}
The divisibility assumption on the index in Theorem \ref{thm:results_initial_form} can be removed on the expense of replacing sharply $F$-pure by strongly $F$-regular (the positive characteristic equivalent of Kawmata log terminal singularities), as stated in Theorems \ref{thm:index_p_divisible_nef} and \ref{thm:index_p_divisible_pushforward}.
\end{rem}

Point \eqref{itm:results_initial_form:push_forward} of Theorem \ref{thm:results_initial_form} is the sharply $F$-pure version of the characteristic zero statement used to show projectivity of the moduli space of stable varieties \cite{Kollar_Projectivity_of_complete_moduli}, \cite{Fujino_Semi_positivity_theorems_for_moduli_problems}. Therefore, it  implies 
projectivity of coarse moduli spaces of certain sharply $F$-pure moduli functors. For the precise statement we refer the reader to Section \ref{sec:results_full_generality}.

Furthermore, Theorem \ref{thm:results_initial_form} combined with lifting arguments gives a new algebraic proof of the following characteristic zero semi-positivity statement. 

\begin{cor}
Let $f : X \to Y$ be surjective, projective morphism  from a  Kawamata log terminal variety to a smooth projective curve over an algebraically closed field of characteristic zero. Let $r$ be the index of $K_X$.
\begin{enumerate}
\item If $K_{X/Y}$ is $f$-semi-ample, then $K_{X/Y}$ is nef.
\item If $K_{X/Y}$ is $f$-ample, then $f_* \sO_X(mr K_{X/Y})$ is a nef vector bundle for $m \gg 0$.
\end{enumerate}
\end{cor}

\subsection{Results: full generality}
\label{sec:results_full_generality}

In algebraic geometry, one is frequently forced to work with pairs or even with non-normal pairs for various  reasons: induction on dimension, compactification, working with non-proper varieties, etc.  Hence, in the present article we put our results in the following, slightly more general framework than that of Section \ref{sec:results_initial_form}. 

% \begin{notation}
% %\label{notation:results}
% Consider the following situation. Let $f : (X, \Delta) \to Y$ be a flat, relatively $S_2$ and $G_1$, projective morphism  from a pair to a smooth variety $Y$ 
% \begin{itemize}
% \item $K_{X/Y} + \Delta$ is $\bQ$-Cartier,
% \item $p \nmid \ind (K_{X/Y} + \Delta)$,
% \item $\Delta$ contains no components of any fiber,
% \item $K_{X/Y} + \Delta$ is $f$-ample and
% \item there is a sharply $F$-pure fiber,
% \end{itemize}
% 
% 
% \end{notation}

\begin{notation}
\label{notation:results}
Let $f : X \to Y$ be a flat, relatively $S_2$ and $G_1$, equidimensional, projective morphism to a projective scheme over $k$  and $\Delta$ a $\bQ$-Weil divisor on $X$, such that 
\begin{enumerate}
\item  $\Supp \Delta$  contains neither codimension 0 points nor singular codimension 1 points of the fibers,
\item there is a $p \nmid r>0$, such that $r \Delta$ is an integer divisor, Cartier in relative codimension 1 and $\omega_{X/Y}^{[r]}(r \Delta)$ is a line bundle (note that $\omega_{X/Y}^{[r]}(r \Delta)$ is defined as  $ \iota_* (\omega_{U/Y}^r(r \Delta|_U))$  where $\iota : U \to X$ is the intersection of the relative Gorenstein locus and the locus where $r \Delta$ is Cartier) and
\item for all but finitely many $y \in Y$, $(X_y, \Delta_y)$ is sharply $F$-pure (see Definition \ref{defn:non_F_pure}).
\end{enumerate}
\end{notation}
The main results of the paper are as follows.
\begin{thm}
\label{thm:relative_canonical_nef_intro}
In the situation of Notation \ref{notation:results}, if $\omega_{X/Y}^{[r]}(r \Delta)$ is $f$-nef and for all but finitely many $y \in Y$, $K_{X_y} + \Delta_y$ is semi-ample, then $\omega_{X/Y}^{[r]}(r \Delta)$ is nef. 
\end{thm}

\begin{thm}
\label{thm:pushforward_nef_intro}
In the situation of Notation \ref{notation:results}, if  $\omega_{X/Y}^{[r]}(r \Delta)$ is $f$-ample and  $(X_y, \Delta_y)$ is sharply $F$-pure for all $y \in Y$, then $f_* ( \omega_{X/Y}^{[mr]}(mr \Delta) )$ is nef for all $m \gg 0$. 
\end{thm}

\begin{rem}
The assumptions  $p \nmid \ind(K_X+ \Delta)$ in the above theorems can be dropped, by replacing sharply $F$-pure by strongly $F$-regular. This is worked out in Theorems \ref{thm:index_p_divisible_nef} and \ref{thm:index_p_divisible_pushforward}.
\end{rem}

Contrary to Theorem \ref{thm:relative_canonical_nef_intro}, in Theorem \ref{thm:pushforward_nef_intro} we assumed that all fibers are sharply $F$-pure. In fact, if the base is a projective curve, then we may assume sharp $F$-purity only for a general fiber as stated in Theorem \ref{thm:pushforward_nef}. However, for higher dimensional bases we cannot assume only sharp $F$-purity of the general fiber. Indeed, there is no reason to expect nefness on a curve of $Y$ over which the singularities are not sharply $F$-pure. This of course does not explain why  we cannot allow finitely many fibers with non sharply $F$-pure fibers in Theorem \ref{thm:pushforward_nef_intro}. 
The reason for the latter, is some technical deficiency of the current theory (in Proposition \ref{prop:H_0_equals_S_0_relative} sharp $F$-purity of all fibers is assumed), which will most likely be overcome soon. On the other hand, if only the general fiber is required to be sharply $F$-pure,  one can still try to prove weak-positivity of $f_* \left( \omega_{X/Y}^{[mr]}(mr \Delta) \right)$. This issue will 
addressed in later articles.

\begin{cor}
\label{cor:ample}
%Let $f : X \to Y$ be a flat, relatively $S_2$ and $G_1$, equidimensional, projective morphism. Further assume that  there is a $p \nmid r>0$, such that  $\omega_{X/Y}^{[r]}$ is an $f$-ample line bundle and the fibers of $f$ are sharply $F$-pure. 
In the situation of Notation \ref{notation:results}, if $\Delta = 0$, $K_{X/Y}$ is $f$-ample and for every $y \in Y$, $X_y$ is sharply $F$-pure, $\Aut(X_y)$ is finite and  there are only finitely many other $y' \in Y$ such that $X_y \cong X_{y'}$, then $\det \left(f_* \omega_{X/Y}^{[m]} \right)$ is an ample line bundle for all $m \gg 0$ and divisible enough. 
\end{cor}

The author has evidence that taking determinant can be removed from the above corollary. I.e., it can be shown that $f_* \omega_{X/Y}^{[m]}$ is ample as a vector bundle. This issue will be also addressed in upcoming articles. 

In addition to the above statements, the semi-ample assumption in Theorem \ref{thm:relative_canonical_nef_intro} can be dropped on the expense that the index $r$ has to be 1, as stated in the following theorem.

\begin{thm}
\label{thm:relative_canonical_nef_2_intro}
In the situation of Notation \ref{notation:results}, if $\omega_{X/Y}^{[r]}(r \Delta)$ is $f$-nef and $r=1$, then $\omega_{X/Y}( \Delta)$ is nef. 
\end{thm}

For the proofs of the above statements, see Sections \ref{sec:semi_positivity_semi_ample} and \ref{sec:semi_positivity_nef}. Hoping that MMP and the moduli space of stable pairs work in positive characteristics as they do in characteristic zero, one would hope that the divisibility condition of Notation \ref{notation:results} could be removed and the sharply $F$-pure condition could be relaxed to semi-log canonical eventually. Unfortunately, the author has no evidence pro or against this (see Section \ref{sec:questions}).

The following are the main applications. The first one states the existence of a projective coarse moduli space for certain functors of stable varieties. Note that stable varieties are the  higher dimensional analogues of stable curves. According to Corollary \ref{cor:projectivity_of_moduli},  the last step of the general scheme of proving existence of projective coarse moduli spaces initiated in \cite{Kollar_Projectivity_of_complete_moduli} works if the singularities are at most sharply $F$-pure (see Section \ref{sec:projectivity} for details).

\begin{corollary_projectivity}
Let $\sF$ be a subfunctor of 
\begin{equation*}
Y \mapsto  \left. \left\{
\left.\raisebox{25pt}{   \xymatrix{ X \ar[d]_f \\ Y } } \right|
\parbox{280pt}{ \small
 $f : X \to Y$ is a flat, relatively $S_2$ and $G_1$, equidimensional, projective morphism with sharply $F$-pure fibers, such that
 there is a $p \nmid r>0$, for which   $\omega_{X/Y}^{[r]}$ is an $f$-ample line bundle, and $\Aut(X_y)$ is finite for all $y \in Y$  
}
\right\} \right/ \textrm{\small $\cong$ over $Y$}.
\end{equation*}
If $\sF$ admits
\begin{enumerate}
\item  a coarse moduli space $\pi: \sF \to V$, which is a proper algebraic space and 
\item  a morphism $\rho : Z \to \sF$ from a scheme, such that $\pi \circ \rho$ is finite and for the family $g : W \to Z$ associated to $\rho$, $\det \left( g_* \omega_{W/Z}^{[mr]} \right)$ descends to $V$ for every high and divisible enough $m$,
\end{enumerate}
then $V$ is a projective scheme. 
\end{corollary_projectivity}

The second application claims that the above positivity results hold in characteristic zero, assuming Conjecture \ref{conj:semi_log_canonical_reduction} stating that semi-log canonical equals dense sharply $F$-pure type. It should be noted that recently  Fujino gave an unconditional proof of Corollary \ref{cor:characteristic_zero} using Hodge Theory \cite{Fujino_Semi_positivity_theorems_for_moduli_problems}.

\begin{corollary_characteristic_zero}
Let  $(X, \Delta)$ be a  pair over an algebraically closed field of characteristic zero with  $\bQ$-Cartier $K_X + \Delta$ and  $f : X \to Y$  a flat, projective morphism  to a smooth projective curve. Further suppose that there is a $y_0 \in Y$, such that $\Delta$ avoids all codimension 0 and the singular codimension 1 points of $X_{y_0}$, and either 
\begin{enumerate}
 \item $(X_{y_0} , \Delta_{y_0})$ is   Kawamata log terminal, or
 \item $(X_{y_0} , \Delta_{y_0})$ is semi-log-canonical and for every model over a $\bZ$-algebra $A$ of finite type, it satisfies the statement of Conjecture \ref{conj:semi_log_canonical_reduction}. 
\end{enumerate}
Assume also that $K_{X/Y} + \Delta$ is $f$-ample (resp. $f$-semi-ample). 
Then   for $m \gg 0$ and divisible enough, $f_* \sO_X(m(K_{X/Y} + \Delta))$ is a nef vector bundle (resp. $K_{X/Y} + \Delta$ is nef).
\end{corollary_characteristic_zero}

The third application is  a special case of subadditivity of Kodaira-dimension. It states  that in  our special setup, suited for moduli theory, if both the base and the general fiber is of (log-)general type then so is the total space. 

\begin{corollary_subadditivity}
In the situation of Notation \ref{notation:results}, if furthermore $Y$ is an $S_2, G_1$, equidimensional projective variety with $K_Y$ $\bQ$-Cartier and big,  $K_{X/Y} + \Delta$ is $f$-semi-ample and  $K_F + \Delta|_F$ is big for the generic fiber $F$, then $K_X + \Delta$ is  big.
\end{corollary_subadditivity}

\subsection{Idea of the proof}
\label{sec:idea_of_the_proof}

To prove the above mentioned semi-positivity results first we show two general statements, Propositions \ref{prop:semi_positive} and \ref{prop:nef}, about semi-positivity of a line bundle and its pushforward. We consider the following situation, neglecting $\Delta$ at this time. Given a fibration $f : X \to Y$ and a Cartier divisor $N$ on $X$ with certain positivity (e.g., $N - K_{X/Y}$ is nef and $f$-ample), we want to prove positivity of $f_* \sO_X(N)$. One way to approach this problem is to try to find sections of $f_* \sN$, where $\sN:=\sO_X(N)$. For that, notice that for nice $Y$ and generic $y \in Y$, there is an isomorphism $(f_* \sN)_y \to H^0(X_y, \sN)$. So lifting every element of  $(f_* \sN)_y$ to $H^0(Y, f_* \sN)$ is equivalent to lifting every element of $H^0(X_y, \sN)$ to $H^0(X, \sN)$. Fortunately, there is a nice lifting result available for $F$-singularities by Karl Schwede, see Proposition \ref{prop:surjectivity}. This leads us to proving a global generation result for some twist of 
$f_* \sN$ in 
Proposition \ref{prop:generically_globally_generated}, which then implies nefness of the same twist of $f_* \sN$. The next step is to get rid of this twist. For that we use the product trick of Notation \ref{notation:product}, i.e., we apply our global generation result for $n$-times fiber products of $X$ with itself over $Y$. The upshot is that we obtain nefness of  $\bigotimes_{i=1}^n f_* \sN$ twisted by a line bundle. However, the twist is independent of $n$, which yields nefness of $f_* \sN$ itself. This is done in Proposition \ref{prop:semi_positive}. Then one can consider the natural morphism $f^* f_* \sN \to \sN$. If this is surjective enough and $f_* \sN$ is a nef vector bundle, $\sN$ is nef as well. This is Proposition \ref{prop:nef}. 

Having shown the general semi-positivity statements, deducing the semi-positivities of Theorems \ref{thm:relative_canonical_nef_intro}, \ref{thm:pushforward_nef_intro} and \ref{thm:relative_canonical_nef_2_intro} is still a bit of work. The most tricky is Theorem \ref{thm:relative_canonical_nef_intro}, because the index can be an arbitrary integer not divisible by $p$. In the index one case, the rough idea is as follows. We take a very positive Cartier divisor $L$, and we prove by induction on $q>0$ that $qK_{X/Y} + L$ is nef, using the general nefness result mentioned in the last sentence of the previous paragraph. Then, if this holds for all $q>0$, $K_{X/Y}$ has to be nef as well. Unfortunately, this argument brakes down when $ r:=\ind(K_{X/Y}) > 1$. In that case we have to argue by contradiction. We choose a Cartier divisor $B$, which is the pullback of an ample Cartier divisor from $Y$, and we consider the smallest $t>0$, such that $K_{X/Y} + tB$ is nef. Then similarly to the index one case, we prove 
inductively that $q(rK_{X/Y} + (r-1)tB) + L$ is nef for all $q>0$. Therefore, so is $rK_{X/Y} + (r-1)t B$, and then also $K_{X/Y} + \frac{r-1}{r} t B$. However, $\frac{r-1}{r} t < t$, which contradicts the choice of $t$, unless $K_{X/Y}$ was nef originally. Unfortunately, there is a point where one has to be a bit more careful with this argument: $(r-1)tB$ has to be Cartier. Hence, we cannot really use $t$, we have to use a rational number $\frac{a}{b}$ which is slightly bigger than $t$ and for which $(r-1) \frac{a}{b}$ is integer. However, then the question is whether $\frac{r-1}{r} \frac{a}{b}< t$ is going to hold or not. This is solved at least for $t>1$ by Lemma \ref{lem:number_theory} using elementary number theoretic considerations. Then, by pulling back our family $X \to Y$ via an adequate finite map $Y' \to Y$, we can maneuver ourselves into a situation where $t >1$.

\subsection{Notation}
\label{sec:notation}

We fix an algebraically closed  base field $k$ of positive characteristic $p$.  Every scheme is taken over this base field, and is assumed to be separated and noetherian. The generic point of a subvariety $W$ of a scheme $X$ is denoted by $\eta_W$. For any scheme $X$ over $k$, $X_{\sing}$ denotes the (reduced) closed set of $X$, where $X$ is not regular. 

In the present article, many schemes are not normal. For every such $X$ we consider  only Weil divisors that have no components contained in $X_{\sing}$. By abuse of notation, a \emph{Weil divisor} will always mean such a special divisor. They form a free $\bZ$-module under addition, which we denote by $\Weil^*(X)$. A \emph{Weil divisorial sheaf}, on an $S_2,G_1$ scheme $X$ is a  rank one reflexive subsheaf of the total space of fractions $\sK(X)$. In the present article every  Weil divisorial sheaf is invertible in codimension one, hence by the abuse of notation \emph{Weil divisorial sheaves} will mean Weil divisorial sheaves that are invertible in codimension one. The usual reference for such sheaves is \cite{Hartshorne_Generalized_divisors_on_Gorenstein_schemes}, where they are called almost Cartier divisors. It is important to note that every reflexive sheaf on an $S_2, G_1$ scheme which is invertible in codimension one can be given a Weil divisorial sheaf structure. That is, one can find an embedding of 
it into $\sK(X)$. For every $E=\sum a_D D \in \Weil^*(X)$ one can associate a Weil divisorial sheaf:
\begin{equation}
\label{eq:almost_Cartier}
 \sO_X(-E):=\{f \in \sK(X) |  \forall D:  \ord_{D} f \geq a_D  \}, 
\end{equation}
where if $D \subseteq X_{\sing}$, then $a_D=0$ necessarily, and then by $\ord_{D} f \geq a_D$ we mean that $f \in \sO_{X, \eta_D}$. 
Linear equivalence of Weil divisorial sheaves is defined by multiplying with an invertible element of $\sK(X)$ and addition by multiplying them together and taking reflexive hull. That is, $\sL \sim \sK$ if and only if there is a $f \in \sK(X)^{×}$, such that $\sL =  f \cdot \sK$ and $\sL + \sK = (\sL \cdot \sK)^{**}$. Weil-divisorial sheaves modulo linear equivalence form a group under the above addition, which is denoted by $\Pic^* (X)$.  

If $X$ is of finite type over $k$ and reduced, then there are only finitely many divisors contained in $X_{\sing}$. Hence, by \cite[Proposition 2.11.b]{Hartshorne_Generalized_divisors_on_Gorenstein_schemes}, every Weil divisorial sheaf is linearly equivalent to one of the form \eqref{eq:almost_Cartier}. Therefore the construction of \eqref{eq:almost_Cartier} yields a surjective homomorphism $\Weil^*(X) \to \Pic^*(X)$. One can show that the kernel  consists of the $\sum a_D D \in \Weil^* (X)$, associated to an $f \in \sK^{×}$ (that is, $a_D= \ord_D f$ for divisors $D$ not contained in $X_{\sing}$ and $f$ generates $\sO_{X,\eta_D}$ otherwise). Summarizing, $\Pic^*(X)$ is isomorphic to $\Weil^*(X)$ modulo linear equivalence of Weil divisors. In the current article we mostly use the latter representation of $\Weil^*(X)$. 

Similarly as above, a \emph{$\bQ$-divisor} means a formal sum of  codimension one points not contained in $X_{\sing}$ with rational coefficients.  A $\bQ$-divisor $D$ is \emph{effective}, i.e., $D \geq 0$, if all its coefficients are at least zero. For a $\bQ$-divisor $D$, $\ind(D)$ is the smallest integer $n$, such that $nD$ is an integer Cartier divisor. Given a flat projective morphism $f : X \to Y$ with $X$ being $S_2$ and $G_1$ and $Y$  Gorenstein, $\omega_X = \omega_{X/Y} \otimes f^* \omega_Y$ \cite[Lemma 4.10]{Patakfalvi_Components_of_the_moduli_space_of_stable_schemes}. Hence both $\omega_X$ and $\omega_{X/Y}$ are Weil divisorial sheaves \cite[Corollary 5.69]{Kollar_Mori_Birational_geometry_of_algebraic_varieties}, \cite[Theorem 1.9]{Hartshorne_Generalized_divisors_on_Gorenstein_schemes}. Any of their representing Weil divisors are denoted by $K_X$ and $K_{X/Y}$.

% For a scheme $X$ and an arbitrary coherent sheaf $\sF$, the $n$-th \emph{reflexive power} is
% \begin{equation*}
% \sF^{[n]} :=
% \left\{
% \begin{matrix}
% (\sF^{\otimes n})^{**}  & \textrm{ if } n \geq 0  \\
% (\sF^{\otimes (-n)})^* & \textrm{ if } n < 0 .
% \end{matrix}
% \right.
% \end{equation*}
% That is, it is the reflexive hull of the $n$-th tensor power. A reflexive sheaf  $\sF$ is a \emph{$\bQ$-line bundle} if  $\sF^{[r]}$ is a line bundle for some $r>0$.

%If $Y$ is not Gorenstein, but there is a $\bQ$-divisor $\Delta$ on $X$ and $r>0$, such that $r \Delta$ is integer and  $\omega_{X/Y}^{[r]}(r \Delta)$ is a line bundle, then by $K_{X/Y}$ we mean the $\bQ$-divisor $D/r - \Delta$, where $D$ is a Cartier divisor such that $\sO_X(D) \cong \omega_{X/Y}^{[r]}(r \Delta)$.  One can show that up to $\bQ$-linear equivalence this divisor is independent of the choice of $r$. Furthermore, in the overlapping case, i.e. when the base is Gorenstein and $\omega_{X/Y}^{[r]}$ is a line bundle, then the two definitions of $K_{X/Y}$ agree (up to $\bQ$-linear equivalence). 

\emph{Vector bundle} means a locally free sheaf of finite rank. \emph{Line bundle} means a locally free sheaf of rank one. When it does not cause any misunderstanding, pullback is denoted by lower index. E.g., if $\sF$ is a sheaf on $X$, and $X \to Y$ and $Z \to Y$ are morphisms, then $\sF_Z$ is the pullback of $\sF$ to $X ×_Y Z$. This unfortunately is also a source of some confusion: $\sF_y$ can mean both the stalk and the fiber of the sheaf $\sF$ at the point $y$. Since both are frequently used notations in the literature, we opt to use both and hope that it will always be clear from the context which one we mean.

There are some important conventions of orders of operations, since expressions as $F^e_* \omega_X(\Delta) \otimes \sL$ are used frequently. \textbf{Push-forward has higher priority than tensor product, but twisting with a divisor has higher priority than push-forward.} E.g., the above expression means $(F^e_* (\omega_X(\Delta))) \otimes \sL$.

\subsection{Organization}

Section \ref{sec:semi_positivity} is the core of the article. It contains the details of the argument outlined in Section \ref{sec:idea_of_the_proof}. The necessary definitions, background material and technical statements can be found in Section \ref{sec:background}. Many of these technical statements involve finding uniform bounds for certain behaviors experienced in the presence of  very positive line bundles. Section \ref{sec:applications} contains the proofs of the applications of the general positivity statements. Finally, we collected some of the many questions the article brings up in Section \ref{sec:questions}.

\subsection{Acknowledgement}

The author of the article would like to thank both Karl Schwede and Kevin Tucker for organizing the AIM workshop ``Relating test ideals and multiplier ideals''. The main idea of the article came up when reading about $F$-singularities as a preparation for the workshop. He is especially thankful to Karl Schwede, who answered hundreds of questions both personally when he inivited the author to visit PSU and also in e-mail afterwards. Further the author would also like to thank him for reading a preprint version of the article. 

The author is also thankful to J\'anos Koll\'ar for reading a preprint of the article and for the useful, important advice, especially about the direction in which the initial statements should be developed further. He is further thankful to Chenyang Xu for spotting some errors in the preprint version of the article, Yifei Chen for finding a wrong citation and S\'andor Kov\'acs for the useful remarks. It should also be mentioned that many ideas of the article about how to prove semi-positivity statements stem from the works of Eckart Viehweg. The author never had the opportunity to talk to him personally, but he is greatly thankful for the mathematics he learnt reading the works of Eckart Viehweg. 

\section{Background}
\label{sec:background}

This section contains the necessary technical definitions and statements used in the arguments outlined in Section \ref{sec:idea_of_the_proof} and worked out in Section \ref{sec:semi_positivity}.

\subsection{Definitions}
\label{sec:definitions}

Here we present the definitions used in the paper from the theory of $F$-singularities. As mentioned in the introduction, these are characteristic $p$ counterparts of the singularities of the minimal model program. We give only the minimally needed definitions, we refer the reader to \cite{Schwede_Tucker_A_survey_of_test_ideals} for a general survey on the theory of $F$-singularities. As the singularities of the minimal model, the $F$-singularities show up naturally in lifting statements. This is how they appear in the present article. 

\begin{defn}
\label{defn:pair}
A \emph{pair} $(X,\Delta)$ is an $S_2, G_1$, noetherian, separated scheme over $k$ of pure dimension, with an effective  $\bQ$-divisor $\Delta$.  Note that, according to Section \ref{sec:notation}, a $\bQ$-divisor is a formal sum with rational coefficients of codimension one points that are not contained in the singular locus of $X$. The index $\ind (X, \Delta)$ is defined as  $\ind (K_X + \Delta)$, that is, the smallest positive integer $r$ such that $r(K_X + \Delta)$ is an integer Cartier divisor.
\end{defn}

\begin{notation}
\label{notation:Grothendieck_trace_Frobenius}
Let $(X,\Delta)$ be a pair, such that $K_X + \Delta$ is $\bQ$-Cartier and $p \nmid \ind(X,\Delta)$. Set $g:=\min \{ e \in \bZ^{>0} | (p^e -1)(K_X + \Delta) \textrm{ is Cartier} \}$. (Note that there is an integer $e>0$ for which $(p^e -1)(K_X + \Delta)$ is Cartier, by the index assumption and Euler's theorem. Hence $g$ exists and furthermore, these integers are exactly the multiples of $g$) For any $e \geq 0$ such that $g | e$, define
\begin{equation*}
\sL_{e,\Delta}:= \sO_X((1-p^e)(K_X + \Delta)). 
\end{equation*}
Let $F^e : X \to X$ be the $e$-th iteration of the \emph{absolute Frobenius} morphism, i.e., the map, which is the identity on points and is $r \mapsto r^{p^e}$ on the structure sheaf. Denote by $\phi^e : F^e_* \sL_{e,\Delta} \to \sO_X$ the unique extension from the Gorenstein locus  of the composition of following maps:
\begin{itemize}
\item the embedding $F^e_*  \sO_X((1-p^e)(K_X + \Delta)) \to F^e_* \sO_X((1-p^e)K_X )$ induced by $\sO_X((1-p^e)\Delta) \to \sO_X$ (note that twisting with a divisor has higher priority than pushforward according to Section \ref{sec:notation}) and
\item the twist $F^e_* \sO_X((1-p^e)K_X) \to \sO_X$ of the Grothendieck trace by $\omega_X^{-1}$.
\end{itemize}
\end{notation}

\begin{prop}
\label{prop:decreasing_chain_non_F_pure}
In the situation of Notaion \ref{notation:Grothendieck_trace_Frobenius},
\begin{equation}
\label{eq:decreasing_chain_non_F_pure}
\phi^{e'}F^{e'}_* \sL_{e',\Delta} \subseteq \phi^{e''}F^{e''}_* \sL_{e'',\Delta}, \textrm{ for }  e''|e' \textrm{ such that }  g|e'' .
\end{equation}
\end{prop}

\begin{proof}
Let $m:=e'/e''$ and $f:=(m-1) e''$. Then, 
\begin{equation*}
\phi^{e'} F^{e'}_* \sL_{e',\Delta} 
=  
\underbrace{\phi^{e''}   F^{e''}_* \left( \phi^{f} \otimes \id_{\sL_{e'', \Delta}} ( F^{f}_*  \sL_{f,\Delta}  \otimes  \sL_{e'',\Delta} ) \right) }_{\tiny \parbox{270pt}{ $\phi^{e''} \circ F^{e''}_* (\phi^{f} \otimes \id_{\sL_{e'',\Delta}} )  = \phi^{e'}$ up to multiplication by a unit, because of \cite[Lemma 3.9]{Schwede_F_adjunction} and the projection formula (using that $(F^e)^*  \sL \cong \sL^{p^e}$) }}
\subseteq
\phi^{e''} F^{e''}_*   \sL_{e'',\Delta}.
\end{equation*}

\end{proof}

\begin{defn}
\label{defn:non_F_pure}
In the situation of  Notation \ref{notation:Grothendieck_trace_Frobenius}, define the \emph{non-F-Pure ideal} of $(X, \Delta)$ as
\begin{equation*}
 \sigma(X,\Delta) = \bigcap_{e \geq 0  } \phi^{e \cdot g} F^{e \cdot g}_* \sL_{e \cdot g,\Delta} .
\end{equation*}
 According to \cite[Remark 2.9]{Schwede_A_canonical_linear_system}  and Proposition \ref{prop:decreasing_chain_non_F_pure} this intersection stabilizes, that is,
\begin{equation}
\label{eq:non_F_pure:big_e}
 \phi^{e \cdot g} F^{e \cdot g}_* \sL_{e \cdot g, \Delta}  \cong \sigma (X, \Delta), \textrm{ for all } e \gg 0 .
\end{equation}
Also by \cite[Remark 2.9]{Schwede_A_canonical_linear_system}, if $e>0$ is any integer such that $g|e$, then $\sigma(X, \Delta)$ is the unique largest ideal $\sI$ such that 
\begin{equation*}
 \phi^e F^e_* ( \sL_{e,\Delta} \cdot \sJ ) = \sJ.
\end{equation*}
The pair $(X, \Delta)$ is \emph{sharply $F$-pure} if $\sigma(X,\Delta)=\sO_X$.
\end{defn}

In the known counterexamples to Kodaira vanishing (e.g., \cite{Raynaud_Contre_exemple_au_vanishing_theorem}) one finds elements in adjoint linear systems that do not come from some high Frobenius. Hence, a technique to lift sections in positive characteristic is to consider only  sections of adjoint bundles that come from arbitrary high Frobenius. One such collection of sections is the following subgroup of $H^0(X, \sL)$. For the main application, see Proposition \ref{prop:surjectivity}.

\begin{defn}
\label{defn:S_0_F_pure_ideal}
In the situation of Notation \ref{notation:Grothendieck_trace_Frobenius}, if $\sL$ is a line bundle on $X$, then define
\begin{multline}
\label{eq:S_0_F_pure_ideal:definition}
S^0(X, \sigma(X,\Delta)  \otimes \sL)
\\ := \bigcap_{e  \in \bZ^{ \geq 0}} \im( H^0(X, F^{e \cdot g}_* (\sigma(X,\Delta) \otimes  \sL_{e \cdot g,\Delta}) \otimes \sL) \underset{H^0(\phi^e \otimes \sL)}{\longrightarrow} H^0(X, \sigma(X,\Delta) \otimes  \sL)).
\end{multline}
\end{defn}

\begin{rem}
It is very important to stress that $S^0(X, \sigma(X,\Delta)  \otimes \sL)$ depends on $\Delta$ and $\sL$, not only  on $\sigma(X,\Delta)  \otimes \sL$ and not even on $\sigma(X,\Delta)$ and  $\sL$. 
\end{rem}

The following proposition gives a better description of $S^0(X, \sigma(X, \Delta) \otimes \sL)$. It is the one that will be used throughout the article.

\begin{prop}
\label{prop:S_0_F_pure_ideal}
In the situation of Notation \ref{notation:Grothendieck_trace_Frobenius}, if  $\sL$ is a line bundle on $X$ then
\begin{equation}
\label{eq:S_0_F_pure_ideal:statement}
S^0(X, \sigma(X,\Delta)  \otimes \sL) = \bigcap_{e \in \bZ^{\geq 0}} \im( H^0(X, F^{e \cdot g}_* \sL_{e \cdot g, \Delta} \otimes \sL) \underset{H^0(\phi^e \otimes \sL)}{\longrightarrow} H^0(X,  \sL)).
\end{equation}

\end{prop}

\begin{proof}
The difference between the left side of \eqref{eq:S_0_F_pure_ideal:definition} and \eqref{eq:S_0_F_pure_ideal:statement} is that the $\sigma(X,F)$ is omitted from the latter one. Hence the latter is bigger and in particular using \eqref{eq:S_0_F_pure_ideal:definition},
\begin{equation*}
S^0(X, \sigma(X,F) \otimes \sL) \subseteq \bigcap_{e \in \bZ^{\geq 0}} \im( H^0(X, F^{e \cdot g}_* \sL_{e \cdot g,\Delta} \otimes \sL) \to H^0(X,  \sL)).
\end{equation*}
To prove the other inclusion, consider then the following isomorphisms for any $g | e$.
\begin{multline*}
F^e_* (\sigma(X,F) \otimes  \sL_{e,\Delta} ) 
\cong
\underbrace{ F^e_* \left(  (\phi^{e'} F^{e'}_* \sL_{e',\Delta} ) \otimes  \sL_{e,\Delta}   \right) }_{\textrm{by \eqref{eq:non_F_pure:big_e}, for $e' \gg 0$ such that $g | e'$}}
 \cong
\underbrace{F^e_* \left( \phi^{e'} \otimes \id_{\sL_{e, \Delta}}  \left( F^{e'}_* \left(\sL_{e',\Delta} \otimes  \sL_{e,\Delta}^{p^{e'}} \right) \right)  \right)  }_{\textrm{projection formula}}
\\ \cong
F^e_* (\phi^{e'} \otimes \id_{\sL_{e, \Delta}} ( F^{e'}_* \sL_{e + e', \Delta}))
\cong 
F_*^{e}(\phi^{e'} \otimes \id_{\sL_{e, \Delta}})  (F^{e+e'}_* \sL_{e + e', \Delta})
\end{multline*}
This yields a homomorphism
\begin{multline*}
H^0(X, F^{e+e'}_* \sL_{e + e',\Delta} \otimes \sL ) 
\underset{H^0 \left( F^e_* (\phi^{e'} \otimes \id_{\sL_{e, \Delta}}) \otimes \id_{\sL} \right) }{\xrightarrow{\hspace*{3cm}}}
H^0 \left( X, \left( F^e_* (\phi^{e'} \otimes \id_{\sL_{e, \Delta}}) ( F^{e'}_* \sL_{e + e',\Delta}) \right) \otimes \sL \right) 
\\ \cong 
H^0(X, F^e_* (\sigma(X,F) \otimes  \sL_{e,\Delta}) \otimes \sL ).
\end{multline*}
Hence,
\begin{multline*}
\im (H^0(X, F^e_* (\sigma(X,F) \otimes \sL_{e,\Delta}) \otimes \sL) \to  H^0(X, \sigma(X,F) \otimes \sL) )
\\ \supseteq 
\im (H^0(X, F^{e+e'}_* \sL_{e + e',\Delta} \otimes \sL ) \to H^0(X,  \sL) ),
\end{multline*}
and in particular then using \eqref{eq:S_0_F_pure_ideal:definition} again,
\begin{equation*}
S^0(X, \sigma(X,F) \otimes \sL) \supseteq \bigcap_{e \in \bZ^{>0}} \im( H^0(X, F^{e \cdot g}_* \sL_{e \cdot g, \Delta} \otimes \sL) \to H^0(X, \sL)).
\end{equation*}
\end{proof}

\begin{rem}
\label{rem:S_0_F_pure_ideal_stabilization}
In the situation of Definition \ref{defn:S_0_F_pure_ideal}, if $X$ is projective, then $H^0(X,\sL)$ is finite dimensional. Therefore, \begin{multline*}
 S^0(X, \sigma(X,\Delta)  \otimes \sL) 
=
\im \left( H^0(X, F^{e \cdot g}_* (\sigma(X,\Delta) \otimes  \sL_{e \cdot g,\Delta}) \otimes \sL ) \to H^0(X, \sigma(X,\Delta) \otimes \sL)\right)
\\ = 
\im \left( H^0(X, F^{e \cdot g}_* \sL_{e \cdot g,\Delta} \otimes \sL) \to H^0(X,  \sL) \right)  
\end{multline*}
for all $e \gg 0$.
\end{rem}

\subsection{Cohomology and base change}

Here, we list a couple of standard statements about cohomology and base change as a reference for the following sections. We also introduce the product construction in Notation \ref{notation:product}, one of the main tricks of the article.

\begin{lem} \cite[Theorem 12.11]{Hartshorne_Algebraic_geometry}
\label{lem:cohomology_and_base_change_pushforward_zero}
Let  $f : X \to Y$ be a projective morphism over a noetherian scheme, and $\sG$ a coherent sheaf on $X$ flat over $Y$, such that for all $i>0$, $R^i f_* \sG = 0$. Then, the natural morphisms
\begin{equation*}
 f_* \sG \otimes k(y) \to H^0(X_y, \sG)
\end{equation*}
are isomorphisms, and for all $y \in Y$ and $i>0$,
\begin{equation*}
 H^i(X_y, \sG) =0.
\end{equation*}
\end{lem}

\begin{lem} \cite[Theorem 9.9 and Corollary 12.9]{Hartshorne_Algebraic_geometry}
\label{lem:cohomology_and_base_change_fibers_zero}
Let  $f : X \to Y$ be a projective morphism over a noetherian, integral scheme, and $\sG$ a coherent sheaf on $X$ flat over $Y$, such that for all $i>0$, and $y \in Y$,
\begin{equation*}
 H^i(X_y, \sG) =0.
\end{equation*}
Then $R^i f_* \sG=0$ for $i>0$, $f_* \sG$ is locally free, and the natural homomorphisms 
\begin{equation*}
 f_* \sG \otimes k(y) \to H^0(X_y, \sG)
\end{equation*}
are isomorphisms.
\end{lem}

\begin{notation}
\label{notation:product}
For a morphism $f : X \to Y$ of schemes, define
\begin{equation*}
X^{(m)}_Y :=\underbrace{X ×_Y X ×_Y \dots ×_Y X}_{\textrm{$m$
times}} ,
\end{equation*}
and $f^{(m)}_Y : X^{(m)}_Y \to Y$ is the natural induced map. If $\sF$ is a sheaf of $\sO_X$-modules, then 
\begin{equation*}
\sF^{(m)}_Y := \bigotimes_{i=1}^m p_i^* \sF,
\end{equation*}
where $p_i$ is the $i$-th projection $X^{(m)}_Y \to X$. Similarly, if $\Gamma$ is a divisor on $Y$ and $f$ is flat, then 
\begin{equation*}
\Gamma^{(m)}_Y := \sum_{i=1}^m p_i^* \Gamma,
\end{equation*}
In most cases, we omit $Y$ from our notation. I.e., we use $X^{(m)}$, $\Gamma^{(m)}$, $f^{(m)}$ and $\sF^{(m)}$ instead of $X^{(m)}_Y$,  $\Gamma^{(m)}_Y$, $f^{(m)}_Y$ and $\sF^{(m)}_Y$, respectively. 
\end{notation}

\begin{lem}
\label{lem:cohomology_and_base_change_product}
Let $f : X \to Y$ be a projective flat morphism over a noetherian scheme and $\sG$ a coherent sheaf on $X$ flat over $Y$, such that $H^i(X_y, \sG)=0$ for all $i>0$ and $y \in Y$. Then, using Notation \ref{notation:product}, the natural morphisms
\begin{equation}
\label{eq:cohomology_and_base_change_product:locally_free}
f^{(n)}_* (\sG^{(n)} )  \to \bigotimes_{i=1}^n f_* \sG
\end{equation}
and 
\begin{equation}
\label{eq:cohomology_and_base_change_product:base_change}
f^{(n)}_* (\sG^{(n)} ) \otimes k(y) \to H^0(X_y^{(n)}, \sG^{(n)})
\end{equation}
are isomorphisms.
\end{lem}

\begin{proof}
First note, that since both $\sG$ and $X$ are flat over $Y$, $p_i^* \sG$ is flat over $Y$ as well. However then $\sG^{(n)} = \bigotimes_{i=1}^n p_i^* \sG$ is also flat over $Y$. By the assumptions and Künneth formula, $H^i(X_y, \sG^{(n)})=0$ for all $i>0$ and $y \in Y$. Therefore, by applying Lemma \ref{lem:cohomology_and_base_change_fibers_zero} for both $\sG^{(n)}$ and $\sG$, one obtains that $R^i f_*^{(n)}  \sG^{(n)}=0$ and $R^i f_*  \sG=0$ for $i>0$,  $f_* \sG$ is locally free, and that \eqref{eq:cohomology_and_base_change_product:base_change} holds. The following isomorphisms show \eqref{eq:cohomology_and_base_change_product:locally_free}.
\begin{equation*}
 f_*^{(n)}  \sG^{(n)}
\cong
\underbrace{R f_*^{(n)}  \sG^{(n)}}_{R^i f_*^{(n)}  \sG^{(n)}=0 \textrm{ for } i>0}
\cong
\underbrace{\bigotimes_L {}_{i=1}^n R f_*  \sG}_{\textrm{Künneth formula}}
\cong
\underbrace{\bigotimes_L {}_{i=1}^n f_*  \sG}_{R^i f_*  \sG=0 \textrm{ for } i>0}
\cong
\underbrace{\bigotimes_{i=1}^n f_*  \sG}_{f_* \sG \textrm{ is locally free}}
\end{equation*}

\end{proof}

\subsection{Global generation}

Here we prove a Fujita type uniform global generation result for flat families in Proposition \ref{prop:Fujita_type_global_generation}. The main tool  is the relative version of Fujita vanishing, which is also crucial for many other statements of the article.

\begin{thm} \cite[Theorem 1.5]{Keller_Ample_filters_of_invertible_sheaves}
\label{thm:relative_Fujita} {  \rm \bf (Relative Fujita vanishing)}
Let $f : X \to Y$ be a projective morphism over a noetherian scheme, and $\sL$ an $f$-ample line bundle on $X$. Then for all coherent sheaves $\sF$ on $X$ there is an $M>0$, such that for all $m \geq M$, $i>0$ and $f$-nef line bundle $\sK$, 
\begin{equation*}
R^i f_* ( \sF \otimes \sL^m \otimes \sK)= 0. 
\end{equation*}
\end{thm}

\begin{thm} (e.g., \cite[Theorem 1.8.5]{Lazarsfeld_Positivity_in_algebraic_geometry_I})
\label{thm:Castelnuovo_Mumford_regularity}
If on a projective scheme $X$, $\sL$ is a globally generated ample line bundle and $\sF$ a coherent sheaf such that $H^i(X, \sF \otimes \sL^{-i})=0$ for $i>0$, then $\sF$ is globally generated.
\end{thm}

\begin{prop}
\label{prop:Fujita_type_global_generation}
Let $f : X \to Y$ be a projective morphism over a noetherian scheme, and $\sL$ an $f$-ample line bundle on $X$. Then for all coherent sheaves $\sF$  on $X$ flat over $Y$, there is an $M>0$, such that for every $y \in Y$ and $f$-nef line bundle $\sK$, $\sF \otimes \sL^M \otimes \sK|_{X_y}$ is globally generated.
\end{prop}

\begin{proof}
Let $n$ be the biggest dimension of a fiber of $f$. Pick a globally generated ample line bundle $\sA$ on $X$. Using Theorem \ref{thm:relative_Fujita}, fix a $M>0$, such that  for every $f$-nef line bundle $\sK$ on $X$, 
\begin{equation*}
R^i f_* ( \sF  \otimes \sA^{-n} \otimes \sL^M \otimes \sK)= 0. 
\end{equation*}
Then by Lemma \ref{lem:cohomology_and_base_change_product} for all $y \in Y$ and $f$-nef $\sK$, 
\begin{equation*}
 H^i( X_y, \sF  \otimes \sA^{-n} \otimes \sL^M \otimes \sK)= 0. 
\end{equation*}
In particular since $\sA^{n-i}$ is nef for every $i \leq n$: for all $i \leq n$, $y \in Y$ and $f$-nef $\sK$,
\begin{equation*}
 H^i( X_y, \sF  \otimes \sA^{-i} \otimes \sL^M \otimes \sK)= 0.  
\end{equation*}
Therefore, by Theorem \ref{thm:Castelnuovo_Mumford_regularity}, $\sF \otimes \sL^M \otimes \sK|_{X_y}$ is globally generated, which concludes our proof.
\end{proof}

\subsection{Adjunction and surjectivity}

In Proposition \ref{prop:surjectivity} another main ingredient of the article, the lifting statement, is stated for easier reference (see Section \ref{sec:idea_of_the_proof} for explanation, and Proposition \ref{prop:generically_globally_generated} for the main application).

\begin{defn}
In the situation of Notation \ref{notation:Grothendieck_trace_Frobenius}, a subvariety $Z \subseteq X$ is an $F$-pure center, if $(X, \Delta)$ is  sharply $F$-pure at the generic point of $Z$ and if for some (or equivalently all \cite[Proposition 4.1]{Schwede_F_adjunction}) $e>0$, 
\begin{equation}
\label{eq:F_pure_center:definition}
\phi^{e \cdot g}(F^{e \cdot g}_*(\sI_Z \cdot \sL_{e \cdot g,\Delta})) \subseteq \sI_Z.
\end{equation}
Furthermore, if $Z$ is the union of $F$-pure centers, then \eqref{eq:F_pure_center:definition} still holds. In both situations for any $e>0$, $\sL_{e \cdot g,\Delta} \to \sO_X$ descends then to
\begin{equation*}
 \phi^{e \cdot g}(F^{e \cdot g}_*( \sL_{{e \cdot g},\Delta}|_Z)) \subseteq \sO_Z.
\end{equation*} 
This defines a natural $\bZ_{(p)}$-Weil divisorial sheaf, which then defines a $\bQ$-Weil divisorial sheaf: the different of $\Delta$ on $Z$, denoted by $\Delta_Z$ \cite[Definition 5.1]{Schwede_A_canonical_linear_system}, \cite[Definition 4.4]{Patakfalvi_Schwede_Depth_of_F_singularities}. The only situation where $\Delta_Z$ will be used in this article is if $Z$ is an $S_2$, $G_1$ Cartier divisor and $\Delta$ is $\bQ$-Cartier at the codimension one points of $Z$ with index relatively prime to $p$. Then $\Delta_Z$ is the natural restriction $\Delta|_Z$ \cite[Lemma 4.6]{Patakfalvi_Schwede_Depth_of_F_singularities}. Furthermore, if   $(Z,\Delta_Z)$ is sharply $F$-pure, then $\Delta_Z$ is automatically  a divisor in the sense of the current article (c.f., paragraph before \cite[Lemma 4.6]{Patakfalvi_Schwede_Depth_of_F_singularities}). That is,  none of its components are contained in the singular locus of $Z$. In our situations this will always be the case. 
\end{defn}

\begin{prop} 
\label{prop:surjectivity}
In the situation of Notation \ref{notation:Grothendieck_trace_Frobenius}, if $Z \subseteq X$ is the union of $F$-pure centers of $(X, \Delta)$, and $L$ a Cartier divisor, such that $L - K_X - \Delta$ is ample, then there is a commutative diagram as follows with surjective left vertical arrow.
\begin{equation*}
\xymatrix{
S^0(X, \sigma(X,\Delta) \otimes \sO_X(L)) \ar@{^(->}[r] \ar@{->>}[d] & H^0(X, \sO_X(L)) \ar[d] \\% H^0(X, \sigma(X,\Delta) \otimes \sO_X(L)) \ar[d] \\
S^0(Z, \sigma(Z,\Delta_Z) \otimes \sO_X(L)) \ar@{^(->}[r]  & H^0(Z,  \sO_X(L)) \\ %H^0(Z, \sigma(Z,\Delta_Z) \otimes \sO_X(L)) \\
}
\end{equation*}
\end{prop}

\begin{proof}
The statement is shown in \cite[Proposition 5.3]{Schwede_A_canonical_linear_system} for normal $X$. For $S_2$ and $G_1$ $X$, verbatim the same proof works.
\end{proof}

% \begin{prop}
% \label{prop:adjunction}
% If $(X,\Delta + S)$ is a pair such that $S$ is Cartier in codimension one, $S_2, G_1$, and $p \nmid \ind(K_X + \Delta + S)$, then $(\Delta + S)_{\Delta}=0$.
% \end{prop}
% 
% \begin{proof}
% \cite[Proposition 7.2 and Remark 7.3]{SK_FA}
% \end{proof}

\subsection{Fujita type version for $S_0=H_0$}

This section contains Fujita type results on the equality of  $H^0$ with its subgroup $S^0$,  introduced in Definition \ref{defn:S_0_F_pure_ideal}. As the statements of Section \ref{sec:semi_positivity} need an absolute and a relative version as well, both are presented here.

\begin{notation}
\label{notation:ker_Grothendieck_trace_Frobenius}
In the situation of Notation \ref{notation:Grothendieck_trace_Frobenius}, define $\sB^e_{\Delta}:= \ker(F^{e \cdot g}_* \sL_{e \cdot g,\Delta} \to \sO_X)$ for every $e>0$. Fix also an ample line bundle $\sL$ and assume that $X$ is projective over $k$ and $(X,\Delta)$ is sharply $F$-pure. 
\end{notation}

\begin{rem}
Note that the definition $\sB^e_{\Delta}$ makes sense if we replace $\Delta$ by any $\bQ$-Weil divisor $0 \leq \Delta' \sim \Delta$, since then $\ind (K_X + \Delta') = \ind(K_X + \Delta)$. Here $\Delta \sim \Delta'$ means ordinary linear equivalence, not $\bQ$-linear equivalence. That is, it means that $\Delta - \Delta'$ is the divisor of a $f \in \sK(X)^{×}$.
\end{rem}

\begin{lem}
\label{lem:B_vanishing_cohomology_induction}
In the situation of Notation \ref{notation:ker_Grothendieck_trace_Frobenius}, if  for $M>0$ and an $e_0 >0$ 
\begin{equation*}
\sL^{M(p^g-1)} \otimes \sL_{g,\Delta} 
\end{equation*}
is nef and for every $i>0$, $0 < e \leq e_0$, $m \geq M$ and nef line bundle $\sK$
\begin{equation}
\label{eq:B_vanishing_cohomology_induction:vanishing}
H^i(X, \sB^{e}_{\Delta} \otimes \sL^m \otimes \sK) = 0,
\end{equation}
then the same vanishing holds for $e_0$ replaced by $e_0 + 1$. That is, for every $i>0$, $0 < e \leq e_0+1$, $m \geq M$ and nef line bundle $\sK$, \eqref{eq:B_vanishing_cohomology_induction:vanishing} is satisfied.
\end{lem}

\begin{proof}
Start with the commutative diagram 
\begin{equation*}
\xymatrix{
0 & F_*^g \sL_{g,\Delta} \ar[l] \ar[d] & F^{e \cdot g}_* \sL_{e \cdot g,\Delta} \ar[l] \ar[d] \\
0 &\sO_X \ar[l] \ar[d]  &\sO_X \ar@{=}[l] \ar[d] \\
& 0 & 0
},
\end{equation*}
and extend it with $\sB^{*}_{\Delta}$ to
\begin{equation*}
\xymatrix{
& 0 \ar[d] & 0 \ar[d] \\
& \sB^1_{\Delta} \ar[d] & \sB^e_{\Delta} \ar[d] \ar[l] & \sC^e \ar[l] \ar[d] & 0 \ar[l] \\
0 & F_*^g \sL_{g,\Delta} \ar[l] \ar[d] & F^{e \cdot g}_* \sL_{e \cdot g,\Delta} \ar[l] \ar[d] & F_*^g (\sB^{e-1}_{\Delta} \otimes \sL_{g,\Delta}) \ar[l]  & 0 \ar[l] \\
0 &\sO_X \ar[l] \ar[d] &\sO_X \ar@{=}[l] \ar[d] & 0 \ar[l] \\
& 0 & 0 
},
\end{equation*}
where $\sC^e$ is defined as the kernel of $\sB^e_{\Delta} \to \sB^1_{\Delta}$. By the snake lemma $\sC^e \to F_*^g ( \sB^{e-1}_{\Delta} \otimes \sL_{g,\Delta})$ is isomorphism and $ \sB^e_{\Delta} \to \sB^1_{\Delta}$ is surjective.  Therefore, there is an exact sequence
\begin{equation*}
\xymatrix{
0 \ar[r] & F_*^g(\sB^{e-1}_{\Delta}  \otimes \sL_{g,\Delta}) \ar[r] & \sB^e_{\Delta} \ar[r] &  \sB^1_{\Delta} \ar[r] & 0  . \\
}
\end{equation*}
If $\sK$ is an arbitrary nef line bundle, then this yields another exact sequence:
\begin{equation}
\label{eq:B_vanishing_cohomology_induction}
\xymatrix{
0 \ar[r] & F_*^g(\sB^{e-1}_{\Delta} \otimes \sL_{g,\Delta} \otimes \sL^{p^g m} \otimes \sK^{p^g}  ) \ar[r]  &  \sB^e_{\Delta} \otimes \sL^m \otimes \sK \ar[r] &  \sB^1_{\Delta} \otimes \sL^m \otimes \sK \ar[r] & 0 , \\
}
\end{equation}
where 
\begin{equation*}
 F_*^g(\sB^{e-1}_{\Delta} \otimes \sL_{g,\Delta} \otimes \sL^{p^g m} \otimes \sK^{p^g}  )  \cong F_*^g(\sB^{e-1}_{\Delta} \otimes \sL^m \otimes (\sL^{M(p^g-1)} \otimes \sL_{g,\Delta}) \otimes (\sL^{(m-M)(p^g-1)} \otimes \sK^{p^g}))
\end{equation*}
Notice that by assumptions both $\sL^{M(p^g-1)} \otimes \sL_{g,\Delta}$ and  $\sL^{(m-M)(p^g-1)} \otimes \sK^{p^g}$ are nef. Therefore, applying cohomology to \eqref{eq:B_vanishing_cohomology_induction} yields the statement of the lemma.
\end{proof}

\begin{lem}
\label{lem:B_1_vanishing_uniform}
In the situation of Notation \ref{notation:ker_Grothendieck_trace_Frobenius}, there is a $M>0$, such that for every $i>0$, $m \geq M$,  nef line bundle $\sK$, and $0 \leq \Delta' \sim \Delta$ such that $(X, \Delta')$ is sharply $F$-pure, 
\begin{equation*}
 H^i (X, \sB^1_{\Delta'} \otimes \sL^m \otimes \sK ) = 0 .
\end{equation*}
\end{lem}

\begin{proof}
Consider the following commutative diagram, where $\sC_{\Delta'}$, $\sD_{\Delta'}$ and $\sE_{\Delta'}$ are defined as the cokernels and kernels of the adequate vertical maps.
\begin{equation}
\label{eq:B_1_vanishing_uniform:commutative diagram}
\xymatrix{
& 0 & 0 \\
& \sC_{\Delta'} \ar[r] \ar[u] & \sE_{\Delta'} \ar[u] \\
0 \ar[r] & \sB^1 \ar[r] \ar[u] & F^g_* \sO_X((1 - p^g) K_X) \ar[r] \ar[u] & \sO_X \ar[r] \ar@{=}[d] & 0 \\
0 \ar[r] & \sB^1_{\Delta'} \ar[r] \ar[u] & F^g_* \sO_X((1 - p^g)( K_X + \Delta')) \ar[r] \ar[u] & \sO_X \ar[r] & 0 \\
& \sD_{\Delta'}  \ar[u] & 0 \ar[u] \\
& 0 \ar[u]
}
\end{equation}
By the Snake-lemma, $\sC_{\Delta'} \cong \sE_{\Delta'}$ and $\sD_{\Delta'}=0$. Hence, for every $0 \leq \Delta' \sim \Delta$ such that $(X, \Delta')$ is sharply $F$-pure, there is an exact sequence
\begin{equation}
\label{eq:B_1_vanishing_uniform:exact}
\xymatrix{
0 \ar[r] & \sB^1_{\Delta'} \ar[r] & \sB^1 \ar[r] & \sE_{\Delta'} \ar[r] & 0 .
}
\end{equation}
Since, $F^g_* \sL_{g, \Delta'}=F^g_* \sO_X((1-p^g)(K_X + \Delta'))$ is isomorphic for every $0 \leq \Delta' \sim \Delta$, by the middle column of \eqref{eq:B_1_vanishing_uniform:commutative diagram}, $\chi(X,\sE_{\Delta'} \otimes \sL^n)$ is independent of $\Delta'$. However, then by \eqref{eq:B_1_vanishing_uniform:exact}, every possible $\sE_{\Delta'}$ is contained  in finitely many components of the quot-scheme $\Quot_{\sB^1}$. Let $\fM$ be the union of these components. There is a universal family 
\begin{equation*}
\xymatrix{
0 \ar[r] & \sF \ar[r] & p^* \sB^1 \ar[r] & \sG \ar[r] & 0
}
\end{equation*}
over $\fM ×X$, where $p : \fM ×X \to X$ is the natural projection. By the choice of $\fM$, for every $\Delta'$ there is a $y \in \fM$, such that $\sB^1 \to \sG_y$ is isomorphic to the surjection $\sB^1 \to \sE_{\Delta'}$ with identity at $\sB^1$. Therefore, it is enough to prove the statement of the lemma for all $\sF_y$ instead of all $\sB_{\Delta'}^1$. However, that is just
Theorem \ref{thm:relative_Fujita} and Lemma \ref{lem:cohomology_and_base_change_pushforward_zero} applied to the morphism $\fM ×X \to X$ and $\sF$. 
\end{proof}

\begin{prop}
\label{prop:uniform_vanishing_for_B_e}
In the situation of Notation \ref{notation:ker_Grothendieck_trace_Frobenius}, there is an $M>0$  such that for every $i>0$, $e>0$, $m \geq M$, nef line bundle $\sK$ and $0 \leq \Delta' \sim \Delta$ such that $(X, \Delta')$ is sharply $F$-pure,
\begin{equation*}
H^i(X, \sB^e_{\Delta'} \otimes \sL^m \otimes \sK) = 0.
\end{equation*}
\end{prop}

\begin{proof}
First, since $\Delta' \sim \Delta$, $\sL_{g,\Delta} \cong  \sL_{g,\Delta'}$. In particular,  there is an $M>0$ such that 
\begin{equation*}
\sL^{M(p^g-1)} \otimes \sL_{g,\Delta'} 
\end{equation*}
is nef for all $\Delta'$ as above. Second by Lemma \ref{lem:B_1_vanishing_uniform} after possibly increasing $M$, for every $i>0$, $m \geq M$, nef line bundle $\sK$ and $0 \leq \Delta' \sim \Delta$ such that $(X, \Delta')$ is sharply $F$-pure,
\begin{equation*}
H^i(X, \sB^1_{\Delta'} \otimes \sL^m \otimes \sK) = 0.
\end{equation*}  
Hence the assumptions of Lemma \ref{lem:B_vanishing_cohomology_induction} are satisfied with $e_0=1$ and $\Delta$ replaced by $\Delta'$. Applying Lemma \ref{lem:B_vanishing_cohomology_induction} inductively for all $\Delta'$ at once concludes our proof.
\end{proof}

\begin{cor}
\label{cor:Fujita_type_S_0_equals_H_0}
In the situation of Notation \ref{notation:ker_Grothendieck_trace_Frobenius}, there is an $M>0$  such that for every $m \geq M$, nef line bundle $\sK$ and $0 \leq \Delta' \sim \Delta$ such that $(X, \Delta')$ is sharply $F$-pure,
\begin{equation}
\label{eq:Fujita_type_S_0_equals_H_0}
S^0(X, \sigma(X,\Delta') \otimes \sL^m \otimes \sK) = H^0(X,  \sL^m \otimes \sK).
\end{equation}
\end{cor}

\begin{proof}
Choose $M$ to be the $M$ obtained in Proposition \ref{prop:uniform_vanishing_for_B_e}. Consider the exact sequence
\begin{equation*}
\xymatrix{
0 \ar[r] & \sB^e_{\Delta'} \ar[r] & F^{e \cdot g}_* \sL_{e \cdot g,\Delta'} \ar[r] & \sO_X \ar[r] & 0,
}
\end{equation*}
twist it by $\sL^m \otimes \sK$ and apply $H^i(\_)$ to it. Since for all $m>M$ and $e>0$, $H^1(X, \sB^e_{\Delta'} \otimes \sL^m \otimes \sK) = 0$ by the choice of $M$, we obtain a surjection
\begin{equation*}
H^0(X,F^{e \cdot g}_* \sL_{e \cdot g,\Delta'} \otimes \sL^m \otimes \sK) \twoheadrightarrow H^0(X,\sL^m \otimes \sK).
\end{equation*}
In particular this implies \eqref{eq:Fujita_type_S_0_equals_H_0}.
\end{proof}

\begin{lem}
Let $(X, \Delta)$ be a pair with  $K_X + \Delta$ $\bQ$-Cartier,  $p \nmid \ind(K_X + \Delta)$ and let $f : X \to Y$ be a flat, projective, equidimensional, relative $S_2$ and $G_1$ morphism to a smooth, projective variety of dimension $n$. Assume also that $\Delta$ avoids all the codimension 0 and the singular codimension 1 points of the fibers and that $(X_y, \Delta_y)$ is a sharply $F$-pure pair for all $y \in Y$. Let $A_i \ (i=1, \dots,n)$ be pullbacks of generic hyperplane sections of $Y$. Then $(X, \Delta + \sum_{i=1}^n A_i)$ is sharply $F$-pure as well. 
\end{lem}

\begin{proof}
The question is local, so we may fix $y \in Y$, and then it is enough to prove the statement of the lemma in any open neighborhood of  $X_y$. By reindexing the $A_i$, we may also assume, that $X_y \subseteq A_i$ exactly if $1 \leq i \leq q$. Set then $C_i := A_i$ for $i \leq q$. Furthermore choose, after possibly restricting $Y$, Cartier divisors $C_{q+1}, \dots, C_n$ on $X$ that are pullbacks of smooth divisors on $Y$, such that the equations of $C_1,\dots,C_n$ form a basis of $\factor{m_{Y,y}}{m_{Y,y}^2}$. Applying $F$-inversion of adjunction \cite[Observation 4.5]{Patakfalvi_Schwede_Depth_of_F_singularities} repeatedly and possibly restricting $Y$ along this process, implies that 
\begin{equation*}
\left( \bigcap_{i=1}^j C_i, \left. \Delta + \sum_{l=j+1}^n C_l \right|_{\bigcap_{i=1}^j C_i}  \right)
\end{equation*}
is sharply $F$-pure, where the empty intersection  is defined to be $X$ itself. In particular, for $j=0$ we obtain the statement of the lemma.
\end{proof}

\begin{prop}
\label{prop:H_0_equals_S_0_open_condition}
Let $(X, \Delta)$ be a pair with $K_X + \Delta$ $\bQ$-Cartier,  $p \nmid \ind(K_X + \Delta)$   and let $f : X \to Y$ a flat, projective, equidimensional, relative $S_2$ and $G_1$ morphism to a smooth, projective variety of dimension $n$. Assume also that $\Delta$ avoids all the codimension 0 and the singular codimension 1 points of the fibers and that $(X_y, \Delta_y)$ is a sharply $F$-pure pair for all $y \in Y$. If $\sL$ is  an ample line bundle, then there is an $M>0$ such that for all  $m \geq M$, nef line bundle $\sK$ and generic $y \in Y$,
\begin{equation*}
 H^0(X_y,  \sL^m \otimes \sK ) = S^0(X_y, \sigma(X_y,\Delta_y) \otimes ( \sL^m \otimes \sK)_y )  .
\end{equation*}
\end{prop}

\begin{proof}
Let $A_i \ (i=1, \dots,n)$ and $A_{i,0} \ (i=1, \dots,n)$ be pullbacks of generic hyperplane sections of $Y$. We think about $A_i$ as moving and $A_{i,0}$ as being fixed. Fix $y_0 \in \bigcap_{i=1}^n A_{i,0}$. Choose then $M>0$, such that the statement of Corollary \ref{cor:Fujita_type_S_0_equals_H_0} holds for both $\left( X, \Delta + \sum_{i=1}^n A_i \right)$, $\sL$ and $(X_{y_0}, \Delta_{y_0})$, $\sL_{y_0}$. Let $L$ be a divisor of $\sL$. By possibly increasing $M$, we may also assume that $M L - (K_X + \Delta + \sum_{i=1}^n A_i )$ is ample. 

Consider the following commutative diagram for a $y \in \bigcap_{i=1}^n A_i$, $m \geq M$ and a nef line bundle $\sK$ on $X$. Here the left vertical arrow is the homomorphism  of Proposition \ref{prop:surjectivity} composed $n$-times.
\begin{equation}
\label{eq:H_0_equals_S_0_open_condition:commutative}
\xymatrix{
 S^0(X, \sigma(X, \Delta + \sum_{i=1}^n A_i )  \otimes  \sL^m \otimes \sK ) \ar@{^(->}[r] \ar@{->}[d] & H^0(X,  \sL^m \otimes \sK) \ar@{->}[d] \\
S^0(X_y , \sigma(X_y, \Delta_y) \otimes (\sL^m \otimes \sK)_y ) \ar@{^(->}[r] & H^0(X_y,  \sL^m \otimes \sK) 
}
\end{equation}
By applying Proposition \ref{prop:surjectivity} $n$-times, one obtains that the left vertical arrow is surjective. The top horizontal arrow is isomorphism by the choice of $M$.

Consider now the particular case of \eqref{eq:H_0_equals_S_0_open_condition:commutative}, when $y=y_0$ and $A_i = A_{i,0}$. Then by the choice of $M$, the bottom horizontal arrow is isomorphism as well. Hence, the right vertical arrow has to be surjective as well. However, then it is surjective for every $y$ in a neighborhood of $y_0$ \cite[Theorem III.12.11.a]{Hartshorne_Algebraic_geometry}. But then using the surjectivity of the left and top arrows for generic $A_i$ and $y \in \bigcap_{i=1}^n A_i$ one obtains that the bottom arrow is isomorphism for generic $y \in Y$. 
\end{proof}

\begin{cor}
\label{cor:H_0_equals_S_0_open_condition}
In the situation of Notation \ref{notation:results}, if $\sL$ is an $f$-ample line bundle on $X$, $Y$ is projective and the pair $(X_y, \Delta_y)$ is sharply $F$-pure for all $y \in Y$, then there is an $M>0$ such that for all  $m \geq M$, nef line bundle $\sK$ and generic $y \in Y$ (i.e., contained in a dense open set),
\begin{equation*}
 H^0(X_y,  \sL^m \otimes \sK ) = S^0(X_y, \sigma(X_y,\Delta_y) \otimes ( \sL^m \otimes \sK)_y ) .
\end{equation*}
\end{cor}

\begin{proof}
First, note that all assumptions of the corollary hold for any projective pullback of the family. (To see that the reflexive power $\omega_{X/Y}^{[r \Delta]}(r \Delta)$ of the relative log-canonical sheaf being a line bundle holds for the pulled back family, show that the corresponding sheaf of the pulled back family is isomorphic to the pull-back of the sheaf of the original family and hence is a line bundle. To show the isomorphism, use that it holds in relative codimension one, and \cite[Corollary 3.7]{Hassett_Kovacs_Reflexive_pull_backs}.) Second, note that proving the statement for any projective pullback via a surjective map $Y' \to Y$ yields the statement for the original family as well. Hence, using a cover $Y' \to Y$ consisting of alterations of the components of $Y$,  we may assume that $Y$ is the union of smooth projective varieties. Note, since the statement of the lemma concerns only the value of $\sL$ on the fibers, we may also replace $\sL$ by $\sL \otimes f^* \sA$, and hence assume that it is 
ample. However 
then the result follows 
from Proposition \ref{prop:H_0_equals_S_0_open_condition}.
\end{proof}

\begin{prop}
\label{prop:H_0_equals_S_0_relative}
In the situation of Notation \ref{notation:results}, if $\sL$ is an $f$-ample line bundle on $X$, $Y$ is projective and the pair $(X_y, \Delta_y)$ is sharply $F$-pure for all $y \in Y$, then there is an $M>0$ such that for all  $m \geq M$, nef line bundle $\sK$ and all $y \in Y$,
\begin{equation*}
 H^0(X_y,  \sL^m \otimes \sK ) = S^0(X_y, \sigma(X_y,\Delta_y) \otimes ( \sL^m \otimes \sK)_y ) .
\end{equation*}
\end{prop}

\begin{proof} 
By Proposition \ref{prop:H_0_equals_S_0_open_condition}, there is an $M>0$ and an open set $U_1 \subseteq Y$, such that the statement of the proposition holds for all $y \in U_1$, instead of all $y \in Y$. Using Proposition \ref{prop:H_0_equals_S_0_open_condition} again for the pullback of the family over $Y \setminus U_1$, one finds an open set $U_2 \subseteq Y \setminus U_1$ and possibly even bigger $M>0$, such that the statement of the proposition holds also for all $y \in U_2$, and then for all $y \in U_1 \cup U_2$. Iterating this process, by the Noetherian property, one obtains finitely many $U_i \subseteq Y $ as above such that $\bigcup U_i = Y$. This finishes our proof. 
\end{proof}

\subsection{Auxilliary statements about the product construction}

Here we present some statements about the construction of Notation \ref{notation:product}.

\begin{prop}
\label{prop:product_S_r}
Using Notations \ref{notation:product}, if $f : X \to Y$ is a morphism with $Y$ Cohen-Macaulay and $\sF$ a flat  $S_r$ coherent sheaf sheaf on $X$, then $\sF^{(n)}$ is $S_r$ as well. 
\end{prop}

\begin{proof}
Let $g : Z \to W$ be a morphism to a Cohen-Macaulay scheme, and $\sG$ a flat coherent sheaf on $Z$. Then by  \cite[Proposition 6.3.1]{Grothendieck_Elements_de_geometrie_algebrique_IV_II} and \cite[Corollaire 6.1.2]{Grothendieck_Elements_de_geometrie_algebrique_IV_II}, $\sG$ is $S_r$ if and only if for each $Q \in W$, $\sG|_{Z_Q}$ is $S_{r - \dim \sO_{W,Q}}$. Getting back to the situation of our proposition, since $\sF$ is $S_r$, $\sF|_{X_Q}$ is $S_{r - \dim \sO_{Y,Q}}$ for all $Q \in Y$. Then, by \cite[Lemma 4.2]{Patakfalvi_Schwede_Depth_of_F_singularities}, $\sF^{(n)}|_{X^{(n)}_Q}$ is $S_{r - \dim \sO_{Y,Q}}$ for all $Q \in Y$. Therefore, $\sF^{(n)}$ is $S_r$ on $X^{(n)}$.
\end{proof}

\begin{prop}
\label{prop:G_1_for_products}
Using Notations \ref{notation:product}, if $f : X \to Y$ is a  morphism from a $G_1$ scheme to a smooth curve, then $X^{(n)}$ is $G_1$ for any $n$.
\end{prop}
 
\begin{proof}
Notice that for $X$ to be $G_1$, $X_y$ has to be $G_1$ for generic $y$, and $G_0$ for every $y \in Y$. However then for every $n \in \bZ^+$, $X_y^{(n)}$ is $G_1$ for generic $y$, and $G_0$ for every $y \in Y$, which concludes our proof.
\end{proof}

\begin{prop}
\label{prop:relative_canonical_sheaf_product}
Using Notations \ref{notation:product}, if $f : X \to Y$ is a projective, flat  morphism from an $S_2$ and $G_1$ scheme to a smooth curve, then $\omega_{X/Y}^{(n)} \cong \omega_{X^{(n)}/Y}$.
\end{prop}

\begin{proof}
First, notice that both $\omega_{X/Y}$ and $\omega_{X^{(n)}/Y}$ are reflexive. Furthermore, by \cite[Lemma 2.11]{Bhatt_Ho_Patakfalvi_Schnell_Moduli_of_products_of_stable_varieties}, so is $\omega_{X/Y}^{(n)}$. Therefore, it is enough to prove that $\omega_{X/Y}^{(n)} \cong \omega_{X^{(n)}/Y}$ in codimension one. However if $U$ is the relative Gorenstein locus of $f$, then the isomorphism is clear over $U^{(n)}$. This concludes our proof since $\codim_{X^{(n)}} X^{(n)} \setminus U^{(n)} \geq 2$ by Proposition \ref{prop:G_1_for_products}.
\end{proof}

\begin{lem}
\label{lem:product}
In the situation of Notation \ref{notation:Grothendieck_trace_Frobenius}, if $\sL$ is a line bundle on $X$ and $X$ is reduced, then 
\begin{equation*}
S^0 \left( X^{(m)}, \sigma \left( X^{(m)}, \Delta^{(m)} \right) \otimes \sL^{(m)} \right) \cong S^0(X, \sigma(X, \Delta) \otimes \sL)^{\otimes m}.
\end{equation*}
(Here $X^{(m)}$ and $\sL^{(m)}$ are taken over $\Spec k$.)
\end{lem}

\begin{proof}
A word of caution before starting the proof: $S^0 \left( X^{(m)}, \sigma(X^{(m)}, \Delta^{(m)}) \otimes \sL^{(m)} \right)$ to be defined, $(X^{(m)}, \Delta^{(m)})$ has to be a pair. That is, $X^{(m)}$ has to be $S_2$ and $G_1$, and $\Delta^{(m)}$ has to be an element of $\Weil^*(X)$, i.e., none of the components of $\Delta^{(m)}$ can be contained in $\left( X^{(m)} \right)_{\sing}$. The conditions on $X^{(m)}$ follow from Propositions \ref{prop:product_S_r} and \ref{prop:G_1_for_products}. For the condition on $\Delta^{(m)}$, notice that since $X$ is reduced and we are working over an algebraically closed field, the components of $X$ are generically smooth. It is immediate then that all generic points of $\Delta^{(m)}$ are contained in the smooth locus of $X^{(m)}$. 

After the preliminary considerations note that for every $e>0$ such that $g|e$,
\begin{equation*}
(1-p^e)(K_{X^{(m)}} + \Delta^{(m)})= \sum_{i=1}^m p_i^* ((1-p^e)(K_X + \Delta))
\end{equation*}
over the Gorenstein locus and then by codimension argument it holds everywhere. 
 Hence for every $e >0$, such that  $g|e$,
\begin{equation}
\label{eq:product:L_e}
\sL_{e,\Delta^{(m)}} \cong  \bigotimes_{i=1}^m p_i^* \sL_{e,\Delta}
\end{equation}
therefore, 
\begin{equation}
\label{eq:product:canonical}
F^e_* \sL_{e,\Delta^{(m)}} \cong 
\underbrace{F^e_* \left( \bigotimes_{i=1}^m p_i^* \sL_{e,\Delta} \right)}_{\textrm{\eqref{eq:product:L_e}}} \cong
\underbrace{ \left( \bigotimes_{i=1}^m F^e_* p_i^* \sL_{e,\Delta} \right)}_{\textrm{property of Frobenius}} \cong
\underbrace{\bigotimes_{i=1}^m p_i^* F^e_*  \sL_{e,\Delta} .}_{\textrm{flat base-change}}
\end{equation}
Furthermore, the induced trace maps also respect this decomposition. Hence the following commutative diagram concludes the proof.
\begin{equation*}
\xymatrix{
H^0(X, F^e_* \sL_{e,\Delta} \otimes  \sL)^{\otimes m} \ar[r] \ar@{<->}[d]^{\cong \textrm{ by \eqref{eq:product:canonical}}} & H^0(X,  \sL)^{\otimes m} \ar@{<->}[d]^{\cong}  \\
H^0 \left(X^m, F^e_* \sL_{e,\Delta^{(m)}} \otimes  \sL^{(m)} \right) \ar[r] & H^0 \left( X^{(m)},  \sL^{(m)} \right)
}
\end{equation*}

\end{proof}

\section{Semi-positivity}
\label{sec:semi_positivity}

In this section we present the main results of the article. For an outline of the arguments see Section \ref{sec:idea_of_the_proof}.

\subsection{Generic global generation}

\begin{notation}
\label{notation:semi_positivity}
We use the following notations in this section
\begin{enumerate}
\item $(X,\Delta)$ is a pair, such that $K_X + \Delta$ is $\bQ$-Cartier and $p \nmid \ind (K_X + \Delta)$,
\item $f : X \to Y$ is a flat, projective morphism  to a smooth projective curve, 
\item fix also a closed point $y_0 \in Y$ such that  $X_0:= X_{y_0}$ is $S_2$, $G_1$ and reduced and $\Delta$ avoids all  codimension 0 and the singular codimension 1 points of $X_0$, 
\item set $\Delta_0:= \Delta|_{X_0}$ and $r:= \ind (K_X + \Delta)$.
\end{enumerate}
\end{notation}

\begin{rem}
In the situation of Notation \ref{notation:semi_positivity}, note the following:
\begin{enumerate}
\item if a codimension one point $\xi$ of $X_0$ is singular, then $\Delta$ avoids $\xi$ hence it is Cartier there, otherwise if it is non-singular then  $K_X$ is Cartier at $\xi$ and hence by the index assumption on $K_X + \Delta$, $\Delta$ is $\bQ$-Cartier with index not divisible by $p$, 
\item therefore, at all codimension one points of $\xi$, $\Delta$ can be sensibly restricted,
\item $(X_0,\Delta|_{X_0})$ is a pair and
\item by \cite[Lemma 4.6]{Patakfalvi_Schwede_Depth_of_F_singularities}, $\Delta|_{X_0}$ agrees with the $F$-different of $\Delta$ at $X_0$.
\end{enumerate}

\end{rem}

\begin{prop}
\label{prop:generically_globally_generated}
In the situation of Notation \ref{notation:semi_positivity}, choose a Cartier divisor $N$ and set $\sN:=\sO_X(N)$. Assume that
\begin{enumerate}
%\item \label{itm:generically_globally_generated:vanishing} $R^i f_* \sN=0$ for all $i>0$,
\item \label{itm:generically_globally_generated:ample_on_fibers}  $N - K_{X/Y} - \Delta $ is an $f$-ample $\bQ$-divisor,
\item \label{itm:generically_globally_generated:H_0_equals_S_0} $H^0(X_0,\sN)= S^0(X_0, \sigma(X_0, \Delta_0) \otimes \sN|_{X_0})$,
\item \label{itm:generically_globally_generated:nef} $N - K_{X/Y} - \Delta$ is nef.
\end{enumerate}
Then $f_* \sN \otimes \omega_Y(2y_0)$ is generically globally generated.
\end{prop}

\begin{proof}
Set $M:=N + f^* K_Y + 2 X_0$ and $\sM:=\sO_X(M)$. Consider the commutative diagram below. 
\begin{equation}
\label{eq:generically_globally_generated:commutative_diagram}
\xymatrix{
f_* \sM \ar[r] & (f_* \sM) \otimes k(y_0) \ar@{^(->}[r] &  H^0(X_0, \sM) = S^0(X_0, \sigma(X_0,\Delta_0) \otimes \sM|_{X_0}) \\
S^0( X, \sigma(X, \Delta + X_0) \otimes \sM)  \otimes \sO_Y \ar[u] \ar[rru]
} ,
\end{equation}
where $H^0(X_0, \sM) = S^0(X_0, \sigma(X_0,\Delta_0) \otimes \sM)$, because
\begin{equation*}
 H^0(X_0, \sM)  \cong  H^0(X_0, \sN) = S^0(X_0, \sigma(X_0,\Delta_0) \otimes \sN|_{X_0}) \cong S^0(X_0, \sigma(X_0,\Delta_0) \otimes \sM|_{X_0}).
\end{equation*}
Note that 
\begin{equation}
\label{eq:generically_globally_generated:divisors}
M - K_X - \Delta - X_0 
= 
N + f^* K_Y + 2 X_0 - K_{X/Y} - f^* K_Y - \Delta - X_0
= 
N - K_{X/Y} - \Delta + X_0 .
\end{equation}
Note also that $N - K_{X/Y} - \Delta$ is nef by assumption \eqref{itm:generically_globally_generated:nef} and it is relatively ample by \eqref{itm:generically_globally_generated:ample_on_fibers}. Furthermore, $X_0$ is the pullback of an ample divisor from $Y$. Hence, $N - K_{X/Y} - \Delta + X_0$  is ample and then by \eqref{eq:generically_globally_generated:divisors} so is $M - K_X - \Delta - X_0$. Note now the following: 
\begin{itemize}
\item $p \nmid \ind(K_X + \Delta) = \ind (K_X + \Delta +  X_0)$, and
\item since $X_0$ is smooth at all its general points (by reducedness) and $\Delta$ contains no components of $X_0$, $X_0$ is  a union of $F$-pure centers of $(X, \Delta + X_0)$.
%\item $( \Delta + X_0)_{X_0}=\Delta_0$ by Proposition \ref{prop:adjunction}.
\end{itemize}
Hence, Proposition \ref{prop:surjectivity} implies that the diagonal arrow in \eqref{eq:generically_globally_generated:commutative_diagram} is surjective. This finishes our proof.
\end{proof}

\subsection{Semi-positivity general case}

\begin{lem}
\label{lem:generic_global_generation_nef}
If $\sF$ is a vector bundle on a smooth curve $Y$ and $\sL$ is a line bundle such that for every $m>0$, $\left( \bigotimes_{i=1}^m\sF \right) \otimes \sL$ is generically globally generated, then $\sF$ is nef. 
\end{lem}

\begin{proof}
Take a finite cover $\tau : Z \to Y$ by a smooth curve and a quotient line bundle $\sE$ of $\tau^* \sF$. Since $\left( \bigotimes_{i=1}^m\sF \right) \otimes \sL$ is generically globally generated, so is $\left( \bigotimes_{i=1}^m \tau^* \sF \right) \otimes \tau^* \sL$ and hence  $\sE^m \otimes \tau^* \sL$ as well. Therefore $m \deg (\sE) + \deg (\tau^* \sL) \geq 0$ for all $m>0$. In particular then $\deg (\sE) \geq 0$. Since this is true for arbitrary $\tau$ and $\sE$,$\sF$ is nef indeed.
\end{proof}

\begin{lem}
\label{lem:assumptions_for_product}
In the situation of Notation \ref{notation:semi_positivity}, $f^{(n)} : (X^{(n)},\Delta^{(n)}) \to Y$ also satisfies the assumptions of Notation  \ref{notation:semi_positivity}, where we use the product notations introduced in Notation \ref{notation:product}. 
%Furthermore, if $f: X \to Y$ satisfies every assumption of Notation \ref{notation:semi_positivity} except $p \nmid \ind(K_X + \Delta)$ then the $f^{(n)} : (X^{(n)},\Delta^{(n)}) \to Y$ also satisfies every assumptions of Notation  \ref{notation:semi_positivity} except $p \nmid \ind(K_{X^{(n)}} + \Delta^{(n)})$.
\end{lem}

\begin{proof}
We show every assumption of Notation \ref{notation:semi_positivity} for $f^{(n)} : X^{(n)} \to Y$ one by one. 
\begin{itemize}
\item $X^{(n)}$ is $S_2$ and $G_1$ by Propositions \ref{prop:product_S_r} and \ref{prop:G_1_for_products}.
\item Since all the components of $X_0$ are generically smooth, the same holds for a generic fiber, and then one can see that $\Delta^{(n)} \in \Weil^*(X^{(n)})$, in particular, $(X^{(n)},\Delta^{(n)})$ is a pair.
\item $f: X^{(n)} \to Y$  is flat, projective.
\item $X^{(n)}_0$ is reduced, $S_2$ (by \cite[Lemma 4.2]{Patakfalvi_Schwede_Depth_of_F_singularities}) and $G_1$.
\item $\Delta^{(n)}|_{X^{(n)}_{y_0}} = \sum_{i=1}^n p_i^* \Delta_0$ avoids the codimension 0 and the singular codimension one points of $X^{(n)}_{y_0}$, because the same holds for $\Delta_0$ and every component of $X_0$ is generically smooth.
\item  By Proposition \ref{prop:relative_canonical_sheaf_product}, 
\begin{multline*}
\qquad \ind(K_X + \Delta) = \ind(K_{X/Y} + \Delta)  = \ind((K_{X/Y}+ \Delta)^{(n)})
\\ = 
\ind (K_{X^{(n)}/Y} + \Delta^{(n)}) = \ind (K_{X^{(n)}}+ \Delta^{(n)})
\end{multline*}
and hence $K_{X^{(n)}}+ \Delta^{(n)}$ is $\bQ$-Cartier
%. Furthermore, if $p \nmid \ind(K_X + \Delta)$ is assumed, then 
and $p \nmid \ind(K_{X^{(n)}}+ \Delta^{(n)})$.
\end{itemize}
\end{proof}

\begin{prop}
\label{prop:semi_positive}
In the situation of Notation \ref{notation:semi_positivity}, choose a Cartier divisor $N$ and set $\sN:=\sO_X(N)$.
Assume that
\begin{enumerate}
\item \label{itm:semi_positive:vanishing} $R^i f_* \sN=0$ for all $i>0$,
\item \label{itm:semi_positive:ample_on_fibers} $N - K_{X/Y} - \Delta$ is an $f$-ample $\bQ$-divisor,
\item \label{itm:semi_positive:H_0_equals_S_0} $H^0(X_0,\sN)= S^0(X_0, \sigma(X_0, \Delta_0) \otimes \sN|_{X_0})$,
\item \label{itm:semi_positive:nef} $N - K_{X/Y} - \Delta$ is nef.
\end{enumerate}
Then $f_* \sN $ is a nef vector bundle. 
\end{prop}

\begin{proof}
The proof uses the notations introduced in Notation \ref{notation:product}. Let $n>0$ be an integer. First, notice the following:
\begin{itemize}
% \item $X^{(n)}$ is $S_2$ and $G_1$ by Propositions \ref{prop:product_S_r} and \ref{prop:G_1_for_products},
% \item since all the components of $X_0$ are generically smooth, the same holds for a generic fiber, and then one can see that $\Delta^{(n)} \in \Weil^*(X^{(n)})$, in particular, $(X^{(n)},\Delta^{(n)})$ is a pair,
% \item by Proposition \ref{prop:relative_canonical_sheaf_product}, 
% $\ind(K_X + \Delta) = \ind(K_{X/Y} + \Delta)  = \ind((K_{X/Y}+ \Delta)^{(n)}) = \ind (K_{X^{(n)}/Y} + \Delta^{(n)}) $ $= \ind (K_{X^{(n)}}+ \Delta^{(n)})$
% and hence $K_{X^{(n)}}+ \Delta^{(n)}$ is $\bQ$-Cartier, and $p \nmid \ind(K_{X^{(n)}}+ \Delta^{(n)})$,
% \item $f: X^{(n)} \to Y$  is flat, projective,
% \item $X^{(n)}_0$ is reduced, $S_2$ (by \cite[Lemma 4.2]{Patakfalvi_Schwede_Depth_of_F_singularities}) and $G_1$,
% \item $\Delta^{(n)}|_{X^{(n)}_{y_0}} = \sum_{i=1}^n p_i^* \Delta_0$ avoids the codimension 0 and the singular codimension one points of $X^{(n)}_{y_0}$, because the same holds for $\Delta_0$ and every component of $X_0$ is generically smooth,
\item by Lemma \ref{lem:assumptions_for_product}, the assumptions of Notation \ref{notation:semi_positivity} are satisfied for $f^{(n)} : (X^{(n)},\Delta^{(n)}) \to Y$
\item $N^{(n)} - K_{X^{(n)}/Y} - \Delta^{(n)} = (N - K_{X/Y} - \Delta)^{(n)}$ is $f^{(n)}$-ample,
\item by Lemma \ref{lem:product} and the Künneth formula, 
\begin{multline*}
\qquad H^0 \left( X_0^{(n)},\sN^{(n)} \right)
=
H^0(X_0,\sN)^{\otimes n} \\ \cong   S^0(X_0, \sigma(X_0, \Delta_0) \otimes \sN)^{\otimes n}
\cong 
S^0 \left( X_0^{(n)}, \sigma  \left( X_0^{(n)}, \Delta^{(n)} \right) \otimes \sN^{(n)} \right), 
\end{multline*}
\item $N^{(n)} - K_{X^{(n)}/Y} - \Delta^{(n)} = (N - K_{X/Y} - \Delta)^{(n)}$ is nef.
\end{itemize}
Hence Proposition \ref{prop:semi_positive} applies to $(X^{(n)},\Delta^{(n)})$ and $N^{(n)}$, and consequently, $f^{(n)}_*( \sN^{(n)} ) \otimes \omega_Y(2y_0)$ generically globally generated for every $n>0$. 

By assumption \eqref{itm:semi_positive:vanishing} and Lemmas \ref{lem:cohomology_and_base_change_pushforward_zero}, \ref{lem:cohomology_and_base_change_fibers_zero} and \ref{lem:cohomology_and_base_change_product}, $f^{(n)}_*( \sN^{(n)} ) \cong \bigotimes_{i=1}^n f_* \sN$ and $f_* \sN$ is a vector bundle. Therefore, $f_* \sN$ is a vector bundle, such that $\left( \bigotimes_{i=1}^n f_* \sN \right) \otimes \omega_Y(2y_0)$ is generically globally generated for every $n>0$. Hence, by Lemma \ref{lem:generic_global_generation_nef}, $f_* \sN$ is a nef vector bundle. This concludes our proof.

\end{proof}

\begin{prop}
\label{prop:nef}
In the situation of Proposition \ref{prop:semi_positive}, if furthermore
 $\sN$ is $f$-nef and $\sN_y$ globally generated except possibly finitely many $y \in Y$, 
 then $\sN$ is nef.

% furthermore, $y \in Y$, 
% there is a globally generated ample line bundle $B$, such that $N|_{X_y}  - n B - K_{X_y}$ is ample, then $\sN$ is nef.
\end{prop}

\begin{proof}
Consider the following  commutative diagram for every $y \in Y$.
\begin{equation*}
\xymatrix{
f^* f_* \sN \ar[d] \ar[r] &  \sN  \ar[d] \\
H^0(X_y, \sN) \otimes \sO_{X_y} \ar[r] & \sN_y
}
\end{equation*}
The left vertical arrow is surjective because of assumption \eqref{itm:semi_positive:vanishing} of Proposition \ref{prop:semi_positive} and Lemmas \ref{lem:cohomology_and_base_change_pushforward_zero} and \ref{lem:cohomology_and_base_change_fibers_zero}. The bottom horizontal arrow is surjective except finitely many $y \in Y$. Hence $f^* f_* \sN \to \sN$ is surjective except possibly at points lying  over finitely many points of $y \in Y$. To show that $\sN$ is nef, we have to show that $\deg \sN|_C \geq 0$ for every smooth projective curve $C$ mapping finitely to $X$. By assumption this follows if $C$ is vertical. So, we may assume that $C$ maps surjectively onto $Y$. However, then $(f^* f_* \sN)|_C \to \sN|_C$ is generically surjective. Since $f_* \sN$ is nef by Proposition \ref{prop:semi_positive}, so is $f^* f_* \sN$ and hence $\deg (\sN|_C) \geq 0$ has to hold. 
\end{proof}

\subsection{Semi-positivity when the relative log-canonical divisor is relatively semi-ample}
\label{sec:semi_positivity_semi_ample}

\begin{lem}
\label{lem:number_theory}
Let $r>0$ be an integer. If $t>1$ is a real number, then there is a rational number $\frac{a}{b}$, such that $r | b+1$ and 
\begin{equation}
\label{eq:number_theory:goal}
\frac{a}{b+1}< t  < \frac{a}{b}
\end{equation}
Furthermore, both $a$ and $b$ can be chosen to be arbitrarily big.
\end{lem}

\begin{proof}
\eqref{eq:number_theory:goal} is equivalent to 
\begin{equation*}
\frac{b+1}{a} > \frac{1}{t}  > \frac{b}{a} .
\end{equation*}
Since $b+1= c r$ has to hold for some integer $c$,  \eqref{eq:number_theory:goal} together with $r| b+1$ is equivalent to finding a rational number $\frac{c}{a}$, such that
\begin{equation*}
\frac{c}{a} > \frac{1}{tr} > \frac{c}{a} - \frac{1}{ar} ,
\end{equation*}
which is equivalent to finding positive integers $a$ and $c$, such that
\begin{equation*}
c > \frac{1}{tr} a > c - \frac{1}{r},
\end{equation*}
which is equivalent to finding a positive integer $a$, such that
\begin{equation*}
1 > \left\{ \frac{1}{tr} a \right\} > 1 - \frac{1}{r}.
\end{equation*}
However there is such an $a$, since $t>1$ and hence $\frac{1}{tr}< \frac{1}{r}$. 
%Hence, there is a solution, and $a$ can be chosen to be arbitrarily big. However, then by $\frac{b+1}{a} > \frac{1}{t}$, $b$ can also be chosen to be arbitrarily big.
\end{proof}

\begin{lem}
\label{lem:relatively_semi_ample}
If $X \to Y$ is a flat, projective morphism to a smooth, projective curve, $\sL$ is an $f$-nef line bundle on $X$ such that $\sL_y$ is semi-ample for generic $y \in Y$ and $\sA$ an ample line bundle on $Y$,  then $\sL \otimes f^* \sA^l$ is nef for $l \gg 0$.
\end{lem}

\begin{proof}
By \cite[Corollary 12.9]{Hartshorne_Algebraic_geometry} on an open set (depending on $n$) of $Y$, $f_* \sL^n$ is locally free and 
\begin{equation*}
(f_*  \sL^n)_y \to H^0(X_y, \sL^n)
\end{equation*}
is an isomorphism. Therefore, for $n \gg 0$ and divisible enough, $f^* f_* \sL^n \to \sL^n$ is  surjective over an open set of $Y$. Choose $l >0$  such that $(f_* \sL^n) \otimes \sA^{n l}$ is globally generated. Choose also any curve $C$ on $X$. If $C$ is vertical, $\deg \sL \otimes f^* \sA^l|_C \geq 0$ by the assumption that $\sL$ is $f$-nef. Otherwise if $C$ is horizontal, by the choice of $\sA$, $(f^* f_* \sL^n) \otimes f^* \sA^{nl}$ is globally generated. Hence, by the homomorphism ($f^* f_* \sL^n) \otimes f^* \sA^{nl} \to \sL^n \otimes f^* \sA^{nl}$, which is surjective over an open set of $Y$,  $\sL^n \otimes f^* \sA^{nl}|_C$ is generically globally generated. Therefore $\deg \sL \otimes f^* \sA^l|_C \geq 0$. 
\end{proof}

\begin{thm}
\label{thm:relative_canonical_nef}
In the situation of Notation \ref{notation:semi_positivity}, if $(X_0, \Delta_0)$ is sharply $F$-pure, $K_{X/Y}+\Delta$ is $f$-nef and  $K_{X_y} + \Delta_y$ is semi-ample for generic $y \in Y$ (e.g., this is satisfied if $K_{X/Y} + \Delta$ is $f$-ample or $f$-semi-ample), then $K_{X/Y} + \Delta$ is nef.
\end{thm}

\begin{proof}
\emph{Assume that $K_{X/Y} + \Delta$ is not nef.}  Choose any general, effective ample Cartier divisor on $Y$ and let $B$ be its pullback to $X$. By Lemma \ref{lem:relatively_semi_ample} for $s \gg 0$, $K_{X/Y} + \Delta + sB$ is nef. Let then
\begin{equation*}
t:=\min\{0 \leq s  \in \bR | K_{X/Y} + \Delta + sB \textrm{ is nef}\}. 
\end{equation*}
If $t=0$, there is nothing to prove. Hence we may assume $t>0$. If $t \leq 1$, then by taking a degree $d$ smooth cyclic cover\footnote{To be precise, if $B= f^*D$, then we choose a general very ample, effective divisor $H$, such that there is a divisor $G$, for which $D + H \sim dG$. Then $Y':= \Spec_Y \left( \bigoplus_{i=0}^{d-1} \sO_Y(-iG) \right)$, where the ring structure is given by the embedding $\sO_Y(-dG) \hookrightarrow \sO_Y$ using $D + H \sim dG$.}  $Y' \to Y$ of $Y$ for some $d \gg 0$ and $p \nmid d$, and pulling back everything there, we may replace $B$  by a Cartier divisor $B'$, such that $B'd=B$. Indeed, set $X':= X ×_Y Y'$ and let $\pi : X' \to X$ be the natural projection. Then, one has to verify that the pulled back family still satisfies the assumptions of Notation \ref{notation:semi_positivity}: for example $X'$ is $S_2$ because according to  \cite[Corollaire 6.1.2 and Proposition 6.3.1]{Grothendieck_Elements_de_geometrie_algebrique_IV_II}, a flat fibration over a 
smooth curve is $S_2$ exactly if the generic fiber is $S_2$ and the special fibers are $S_1$. This is stable under pullback. One has to be also careful about pulling back $\Delta$. It is not $\bQ$-Cartier, but according to Definition \ref{defn:pair}, none of its components is contained in $X_{\sing}$ and hence $\Delta$  is Cartier outside of a codimension at least two open set. Then one can pull back by pulling back over this open set, and then extending it in the unique way. The only trap in this process is that a priori one of the components of $\pi^* \Delta$ can end up in the singular locus. However, this cannot happen, since outside of finitely many general fibers, $\pi$ is \'etale, and $\Delta$ does not contain any components of the branched (and therefore general) fibers by assumption. The other conditions are immediate\footnote{Remember that since $D$ and $H$ are general, $\pi$ is \'etale over $X_0$.}. 

After applying the above pullback and replacing $\pi^*B$ by $B'$, $t$ changes to $d \cdot t$, because: 
\begin{equation*}
K_{X'/Y} + \pi^* \Delta + d \cdot t B'  = \pi^* ( K_{X/Y} + \Delta  + t B ) 
\end{equation*}
is nef, and   if there was a $s < dt$, such that $K_{X'/Y} + \pi^* \Delta + s B'$ was nef, then for every curve $C$ on $X'$,  $K_{X'/Y} + \pi^* \Delta + sB' |_C \geq 0$ would hold. However, then also 
\begin{equation*}
0 \leq K_{X'/Y} + \pi^* \Delta + sB' |_C = \pi^* \left. \left( K_{X/Y} + \Delta  + \frac{s}{d} B \right) \right|_C = \left. K_{X/Y} + \Delta  + \frac{s}{d} B \right|_{\pi_* (C)}.
\end{equation*}
would hold. Since every curve on $X$ is the pushfoward of a curve from $X'$, this would mean that $K_{X/Y} + \Delta  + \frac{s}{d} B$ was nef, which would contradict the definition of $t$.
Hence as claimed, $t$ changes to $d \cdot t$.  Therefore, since $d$ can be chosen to be arbitrary big, we may indeed assume that $t>1$. Furthermore, by Lemma \ref{lem:number_theory}, then there is a rational number $\frac{a}{b}$, such that $r | b+1$ and $\frac{a}{b+1}< t  < \frac{a}{b}$. Set $A:= \frac{a}{b} B$ and $p:=b+1$. 
%Furthermore, we may assume that $m \gg 0$.

By Theorem \ref{thm:relative_Fujita}, Corollary \ref{cor:Fujita_type_S_0_equals_H_0} and Proposition \ref{prop:Fujita_type_global_generation}, there is an ample Cartier divisor $Q$ on $X$ such that for all $i>0$  and f-nef Cartier divisor $L$, 
\begin{equation}
\label{eq:relative_canonical_nef:assumption_relative}
R^i f_* (\sO_X(Q+ L))=0,
\end{equation}
\begin{equation}
\label{eq:relative_canonical_nef:assumption_central_fiber}
  H^0(X_0, \sO_{X_0}(Q + L)) = S^0(X_0, \sigma(X_0,\Delta_0) \otimes \sO_{X_0}(Q + L)) 
\end{equation}
and
\begin{equation}
\label{eq:relative_canonical_nef:global_generation}
\sO_{X_y}(Q + L ) \textrm{ is globally generated for all } y \in Y . 
\end{equation}
We prove by induction that $q(p (K_{X/Y} + \Delta) + (p-1)A) + Q$ is nef for all $q>0$. For $q=0$ the statement is true by the choice of $Q$. Hence, we may assume that we $(q-1)(p (K_{X/Y} + \Delta) + (p-1)A) + Q$ is nef. Now, we verify that the conditions of Proposition \ref{prop:nef} hold for $N:=q(p (K_{X/Y} + \Delta) + (p-1)A) + Q$ and $\sN:= \so_X(N)$. Indeed:
\begin{itemize}
\item $N$ is Cartier by the choice of $p$ and $A$,
\item $R^i f_* \sN = 0 $ for all $i>0$ because of \eqref{eq:relative_canonical_nef:assumption_relative} and that 
\begin{equation}
\label{eq:relative_canonical_nef:f_nef}
N- Q= q(p(K_{X/Y} + \Delta) + (p-1)A) 
\end{equation}
is an $f$-nef Cartier divisor,
\item the $\bQ$-divisor
\begin{equation*}
\qquad N- K_{X/Y} -  \Delta =\Big( (p-1) (K_{X/Y} + \Delta + A) \Big) + \Big((q-1)(p(K_{X/Y} + \Delta) + (p-1)A) + Q \Big)
\end{equation*}
is not only $f$-ample, but also nef, because of the inductional hypothesis and that $A= \frac{a}{b}B \geq tB$,
\item  using \eqref{eq:relative_canonical_nef:assumption_central_fiber} and the $f$-nefness of \eqref{eq:relative_canonical_nef:f_nef}, 
\begin{equation*}
H^0(X_0, \sN) = S^0(X_0, \sigma(X_0, \Delta_0) \otimes  \sN),
\end{equation*}
\item since all the summands of $N$ are $f$-nef, so is $N$,
\item $N|_{X_y}$ is globally generated, by  \eqref{eq:relative_canonical_nef:global_generation} and since $N - Q$ is $f$-nef according to \eqref{eq:relative_canonical_nef:f_nef}.
\end{itemize}
Hence Proposition \ref{prop:nef} implies that $N$ is nef. This finishes our inductional step, and hence the proof of the nefness of $q(p (K_{X/Y} + \Delta) + (p-1)A) + Q$ for every $q>0$. However, then $p( K_{X/Y} + \Delta) + (p-1) A$ has to be nef as well. Therefore, so is 
\begin{multline*}
 \frac{1}{p} (p (K_{X/Y} + \Delta) + (p-1) A) = (K_{X/Y} + \Delta) + \frac{p-1}{p} A \\ = (K_{X/Y} + \Delta) + \frac{p-1}{p} \frac{a}{b} B = (K_{X/Y}+ \Delta) + \frac{a}{b+1} B .
\end{multline*}
However, $\frac{a}{b+1}<t$, which contradicts the definition of $t$. Therefore our assumption was false, $K_{X/Y} + \Delta$ is nef indeed.

\end{proof}

\begin{lem}
\label{lem:pullback_results}
In the situation of Notation \ref{notation:results}, let $\tau : Y' \to Y$ be a finite morphism from a smooth curve and set $X' := X ×_Y Y'$ and $\pi : X' \to X$, $f' : X' \to Y'$ the induced morphisms. Let $\Delta'$  be the pullback of $\Delta$, which is defined by the following procedure. Take the open set $U$ of Notation \ref{notation:results} and notice that $r\Delta|_U$ is a Cartier divisor. Then, pull it back to $\pi^{-1}U$, extend it uniquely over $X'$ and finally divide all its coefficients by $r$. The extension is unique, since $\codim_{X'} X' \setminus \pi^{-1}U \geq 2$. In the above situation we claim that
\begin{equation*}
\pi^* \omega_{X/Y}^{[r]}(r \Delta) \cong \omega_{X'/Y'}^{[r]}(r \Delta') . 
\end{equation*}
In particular, $p \nmid r=\ind (K_{X'/Y'} + \Delta')$. 
\end{lem}

\begin{proof}
Notice that by construction $\pi^* \omega_{X/Y}^{[r]}(r \Delta)$ and $\omega_{X'/Y'}^{[r]}(r \Delta')$ agree over $\pi^{-1}U$, that is, in relative codimension one. Notice also that since $\omega_{X/Y}^{[r]}(r \Delta)$ is assumed to be a line bundle, so is $\pi^* \omega_{X/Y}^{[r]}(r \Delta)$, and therefore $\pi^* \omega_{X/Y}^{[r]}(r \Delta)$ is reflexive in the sense of \cite{Hassett_Kovacs_Reflexive_pull_backs}. On the other hand, since $\omega_{X'/Y'}^{[r]}(r \Delta')$ is defined as a pushforward of a line bundle from relative codimension one, it is reflexive by \cite[Corollary 3.7]{Hassett_Kovacs_Reflexive_pull_backs}. Therefore by \cite[Proposition 3.6]{Hassett_Kovacs_Reflexive_pull_backs}, $\pi^* \omega_{X/Y}^{[r]}(r \Delta)$ and $\omega_{X'/Y'}^{[r]}(r \Delta')$ are isomorphic everywhere. In particular, then $\omega_{X'/Y'}^{[r]}(r \Delta')$ is a line bundle.
\end{proof}

\begin{proof}[Proof of Theorem \ref{thm:relative_canonical_nef_intro}]
Choose any curve $C$ on $X$. We are supposed to prove that $\deg \left. \omega_{X/Y}^{[r]}(r \Delta) \right|_C \geq 0$. If $C$ is vertical this is immediate since $\omega_{X/Y}^{[r]}(r \Delta)$ is assumed to be $f$-nef. Otherwise, let $Y'$ be the normalization of $C$. It is enough to prove that $\deg \left. \omega_{X/Y}^{[r]}(r \Delta) \right|_{Y'} \geq 0$. Let $\tau : Y' \to Y$ be the induced morphism. Then, we are in the situation of Lemmma \ref{lem:pullback_results}. Using the notation introduced there, it is enough to prove that $\pi^* \omega_{X/Y}^{[r]}(r \Delta)$ is nef. However, according to Lemma \ref{lem:pullback_results}, this is the same as proving that $\omega_{X'/Y'}^{[r]}(r \Delta')$ is nef.  But, $f' : (X', \Delta') \to Y'$ satisfy all assumptions of Theorem \ref{thm:relative_canonical_nef}. Therefore, $ \omega_{X'/Y'}^{[r]}(r \Delta')$ is nef indeed. 
\end{proof}

\begin{thm}
\label{thm:pushforward_nef}
In the situation of Notation \ref{notation:semi_positivity}, assume that $(X_0,\Delta_0)$ is sharply $F$-pure  and $K_{X/Y}+\Delta$ is $f$-ample. Further choose an $M>0$ such that for all $i> 0$ and $m \geq M$,
\begin{equation}
\label{eq:pushforward_nef:assumption_relative}
R^i f_* (\sO_X(mr(K_{X/Y}  + \Delta)))=0 \textrm{ and}
\end{equation}
\begin{equation}
\label{eq:pushforward_nef:assumption_central_fiber}
  H^0(X_0, \sO_{X_0}(mr(K_{X/Y}  + \Delta) )) = S^0(X_0, \sigma(X_0,\Delta_0) \otimes \sO_{X_0}(mr(K_{X/Y}  + \Delta) )) .
\end{equation}
Then  $f_* (\sO_X(mr(K_{X/Y} + \Delta)))$ is a nef vector bundle for every $m \geq M$. 
\end{thm}

\begin{proof} 
Let $N:= mr(K_{X/Y} + \Delta)$ for some $m \geq M$, and $\sN:=\sO_X(N)$. Then the assumptions of Proposition \ref{prop:semi_positive} for this $N$ are satisfied:
\begin{itemize}
\item  $R^i f_* \sN=0$ for all $i>0$ by \eqref{eq:pushforward_nef:assumption_relative},
\item $N - K_{X/Y} - \Delta$ is an $f$-ample $\bQ$-divisor,
\item $H^0(X_0,\sN)= S^0(X_0, \sigma(X_0, \Delta_0) \otimes \sN|_{X_0})$ by \eqref{eq:pushforward_nef:assumption_central_fiber},
\item  $N - K_{X/Y} - \Delta$ is nef by the nefness of $K_{X/Y} + \Delta$ granted by Theorem \ref{thm:relative_canonical_nef}.
\end{itemize}
Therefore, Proposition \ref{prop:semi_positive} applies indeed and hence $f_* \sN = f_* \sO_X(mr(K_{X/Y} + \Delta))$ is a nef vector bundle for all $m \geq M$. This finishes our proof.
 
\end{proof}

\begin{rem}
Note that by Theorem \ref{thm:relative_Fujita} and Corollary \ref{cor:Fujita_type_S_0_equals_H_0}, there is an $M$ as in the statement of Theorem \ref{thm:pushforward_nef}. 
\end{rem}

\begin{proof}[Proof of Theorem \ref{thm:pushforward_nef_intro}]
First, using Theorem \ref{thm:relative_Fujita} and Proposition \ref{prop:H_0_equals_S_0_relative}, choose an $M>0$ such that  for all $m  \geq M$, $i>0$ and $y \in Y$,  
\begin{equation}
\label{eq:pushforward_nef_intro:assumption_relative}
R^i f_* (\omega_{X/Y}^{[mr]}(mr \Delta))=0 \textrm{ and}
\end{equation}
\begin{equation}
\label{eq:pushforward_nef_intro:assumption_central_fiber}
  H^0(X_y, \sO_{X_y}(mr(K_{X_y}  + \Delta_y) )) = S^0(X_y, \sigma(X_y,\Delta_y) \otimes \sO_{X_y}(mr(K_{X_y}  + \Delta_y) )) .
\end{equation}
Choose now any finite map $\tau : Y' \to Y$ from a smooth, projective curve. We are supposed to prove that $\tau^* f_*( \omega_{X/Y}^{[mr]}(mr \Delta))$ is nef for all $m \geq M$. By Lemma \ref{lem:cohomology_and_base_change_pushforward_zero}, $f_*( \omega_{X/Y}^{[mr]}(mr \Delta))$ commutes with base change for every $m \geq M$.
I.e., using the notations of Lemma \ref{lem:pullback_results}, for every $m \geq M$,
\begin{equation*}
\tau^* f_*( \omega_{X/Y}^{[mr]}(mr \Delta)) \cong  f'_*(  \pi^* \omega_{X/Y}^{[mr]}(mr \Delta)).
\end{equation*}
Therefore, it is enough to prove that the latter is nef. However, again by Lemma \ref{lem:pullback_results}, this is equivalent to proving that $f'_* \omega_{X'/Y'}^{[mr]}(mr \Delta')$ is nef for all $m \geq M$, which is exactly the statement of Theorem \ref{thm:pushforward_nef}. 
\end{proof}

\begin{proof}[Proof of Corollary \ref{cor:ample}]
By Theorem \ref{thm:pushforward_nef_intro} there is a $i_0 >0$ such that for all $i  \geq i_0$, $f_* \omega_{X/Y}^{[ir]}$ is a nef vector bundle. By possibly increasing $i_0$ we may also assume that there is a $d_0 >0$ such that for all $d \geq d_0$,  for all $i \geq i_0$ and all $y \in Y$,
\begin{enumerate}
\item the formation of $f_* \omega_{X/Y}^{[ir]}$ commutes with base-change,
\item $i_0 r K_{X_y}$ is very ample,
\item $\xi^d_y : S^d H^0(X_y, \omega_{X_y}^{[i_0r]}) \to H^0(X, \omega_X^{[di_0r]})$ is surjective and
\item \label{itm:ample:ideal_globally_generated} $\Ker \xi^d_{y}$, which can be identified with $H^0(X_y, \sI_y(d))$ for the ideal $\sI_y$  of the embedding $X_y \hookrightarrow \bP^{h^0 \left( X_y, \omega_{X_y}^{[i_0r]} \right)-1}$, globally generates $\sI_y(d)$ for all $y \in Y$.
\end{enumerate}
Consider then the surjective map
\begin{equation}
\label{eq:ample:surjection}
 S^d \left( f_* \omega_{X/Y}^{[i_0 r]} \right) \twoheadrightarrow f_* \omega_{X/Y}^{[di_0 r]}. 
\end{equation}
The fiber of the kernel of this map at $y$ is $\Ker \xi^d_y$. Hence using assumption \eqref{itm:ample:ideal_globally_generated}, the fiber of the ``classifying map'' of \cite[3.9 Ampleness lemma]{Kollar_Projectivity_of_complete_moduli} associated to the surjection \eqref{eq:ample:surjection} are the sets $\{y' \in Y| X_{y'} \cong X_y\}$. By the assumption of the corollary these sets and hence the fibers of the classifying map are finite. In particular then by \cite[3.9 Ampleness lemma]{Kollar_Projectivity_of_complete_moduli}, using the finiteness of the automorphism groups,  $\det \left(f_* \omega_{X/Y}^{[di_0 r]} \right)$ is ample for all $d \geq d_0$.
\end{proof}

\subsection{Semi-positivity when the relative log-canonical divisor is relatively nef}
\label{sec:semi_positivity_nef}

\begin{thm}
\label{thm:relative_canonical_nef2}
In the situation of Notation \ref{notation:semi_positivity}, if $\ind(X,\Delta)=1$, $X_0$ is sharply $F$-pure  and $K_{X/Y}+\Delta$ is $f$-nef, then $K_{X/Y} + \Delta$ is nef.
\end{thm}

\begin{proof}
Using Theorem \ref{thm:relative_Fujita}, Corollary \ref{cor:Fujita_type_S_0_equals_H_0} and Proposition \ref{prop:Fujita_type_global_generation}, there is an ample enough line bundle $\sL$ on $X$, such that for all $i>0$  and f-nef line bundle $\sK$, 
\begin{equation}
\label{eq:relative_canonical_nef2:assumption_relative}
R^i f_* ( \sL \otimes \sK)=0,
\end{equation}
\begin{equation}
\label{eq:relative_canonical_nef2:assumption_central_fiber}
  H^0(X_0, \sL \otimes \sK) = S^0(X_0, \sigma(X_0,\Delta_0) \otimes (\sL \otimes \sK)|_{X_0}) 
\end{equation}
and
\begin{equation}
\label{eq:relative_canonical_nef2:global_generation}
\sL \otimes \sK |_{X_y} \textrm{ is globally generated for all } y \in Y . 
\end{equation}

Let $L$ be a divisor of $\sL$. We prove by induction that $q(K_{X/Y} + \Delta) + L$ is nef for all $q>0$. For $q=0$ the statement is true by the choice of $L$. Hence, we may assume that we $(q-1)(K_{X/Y} + \Delta) + L$ is nef. Now, we verify that the conditions of Proposition \ref{prop:nef} hold for $N:=q(K_{X/Y} + \Delta) + L$ and $\sN:= \so_X(N)$. Indeed:
\begin{itemize}
\item $N$ is Cartier by the index assumption,
\item $R^i f_* \sN = 0 $ for all $i>0$ because of \eqref{eq:relative_canonical_nef2:assumption_relative} and that $K_{X/Y} + \Delta$ is an $f$-nef Cartier divisor,
\item furthermore, the $\bQ$-divisor
\begin{equation*}
 N- K_{X/Y} -  \Delta = (q-1)(K_{X/Y} + \Delta) + L
\end{equation*}
is not only $f$-ample, but also nef by the inductional hypothesis,
\item using the $f$-nefness of $K_{X/Y} + \Delta$ and \eqref{eq:relative_canonical_nef2:assumption_central_fiber}, 
\begin{equation*}
H^0(X_0, \sN) = S^0(X_0, \sigma(X_0, \Delta_0) \otimes  \sN|_{X_0}),
\end{equation*}
\item since all the summands of $N$ are $f$-nef, so is $N$,
\item for every $y \in Y$,  $N|_{X_y}$ is globally generated by  \eqref{eq:relative_canonical_nef2:global_generation}. 
\end{itemize}
Hence Proposition \ref{prop:nef} implies that $N$ is nef. This finishes our inductional step, and hence the proof of the nefness of $q(K_{X/Y} + \Delta) + L$ for every $q>0$. However, then $ K_{X/Y} + \Delta$ has to be nef as well. This concludes our proof.
\end{proof}

\begin{proof}[Proof of Theorem \ref{thm:relative_canonical_nef_2_intro}]
The proof is identical to that of Theorem \ref{thm:relative_canonical_nef2} after setting  $r:=1$ and using Theorem \ref{thm:relative_canonical_nef2} instead of Theorem \ref{thm:relative_canonical_nef}.
\end{proof}

\subsection{The case of indices divisible by $p$}

Given a pair $(X,\Delta)$ with $p | \ind(K_X + \Delta)$, one can perturb $\Delta$ carefully to obtain another pair $(X,\Delta')$ such that $p \nmid \ind(K_X + \Delta')$. This method can be used to move some of our results to the situation where the index of the log-canonical divisor is divisible by $p$. There is one price to be paid: since the perturbed pair has to still satisfy the adequate sharply $F$-pure assumptions, slightly stronger singularity assumptions have to be imposed on the original pair. The adequate class of singularities is strongly $F$-regular singularities, positive characteristic analogues of Kawamata log terminal singularities. These, contrary to sharply $F$-pure singularities are closed under small perturbations, and furthermore form a subset of sharply $F$-pure singularities. In particular, their perturbations are guaranteed to be sharply $F$-pure. For the definition of strongly $F$-regular singularities we refer to \cite[Definition 2.10]{Schwede_A_canonical_linear_system}. Here, we 
only use the property that given a Cartier divisor $A \geq 0$ and a strongly $F$-regular pair $(X,\Delta)$,  $(X, \Delta + \varepsilon A)$ is strongly $F$-regular as well, for every $0 < \varepsilon \ll 1$.

\begin{lem}
\label{lem:index_divisibility_removal}
Let $(X,\Delta)$ be a pair with a flat morphism to a Gorenstein scheme $Y$, such that $K_X + \Delta$ is $\bQ$-Cartier, but $p | \ind(K_X + \Delta)$. Choose also an effective integer divisor $D$ which is linearly equivalent to $K_{X/Y} + A$ for some Cartier divisor $A$ on $X$. Then $K_X +  \Delta + \frac{1}{p^v -1}(D + \Delta)$ is $\bQ$-Cartier with index not divisible by $p$ for every $v \gg 0$.
\end{lem}

\begin{proof}
First, since $Y$ is Gorenstein, it is enough to show that $K_{X/Y} +  \Delta + \frac{1}{p^v -1}(D + \Delta)$ $\bQ$-Cartier with index not divisible by $p$ for every $v \gg 0$. For proving that, we may choose for $K_{X/Y}$ the representative $D - A$. Let $r$ be an integer such that $r(K_{X/Y} + \Delta)$ is Cartier. Then
%\begin{equation*}
\begin{align*}
\mathclap{\frac{r}{(r,p^v)}(p^v -1) \left(K_{X/Y} +\Delta + \frac{1}{p^v -1}(D + \Delta) \right)}   \\ 
\hspace{100pt} &= \frac{r}{(r,p^v)}(p^v -1) \left( D - A + \Delta + \frac{1}{p^v-1}(D + \Delta) \right) 
\\ & =  r \frac{p^v}{(r,p^v)} (D + \Delta) - \frac{r}{(r,p^v)}(p^v -1) A 
\\ & \sim  r \frac{p^v}{(r,p^v)}  (K_{X/Y} +  \Delta)  + \left( r \frac{p^v}{(r,p^v)}- \frac{r}{(r,p^v)}(p^v -1) \right) A,
\end{align*}
%\end{equation*}
which is Cartier. Furthermore, for $v \gg 0$, $\frac{r}{(r,p^v)}(p^v -1)$ is an integer not divisible by $p$. This concludes our proof.
\end{proof}

\begin{thm}
\label{thm:index_p_divisible_nef}
In the situation of Notation \ref{notation:semi_positivity} by possibly allowing $p|\ind(K_X + \Delta)$, if $(X_0, \Delta_0)$ is strongly $F$-regular, $K_{X/Y}+\Delta$ is $f$-nef and  $K_{X_y} + \Delta_y$ is semi-ample for generic $y \in Y$, then $K_{X/Y} + \Delta$ is nef. 
\end{thm}

\begin{proof}
Fix an ample integer Cartier divisor $A$ on $X$ such that there is an effective integer divisor $D$ linearly equivalent to  $A+ K_{X/Y}$ and furthermore $D$ avoids the codimension 0 and the singular codimension 1 points of $X_0$ as well as the singular codimension 1 points of $X$.  Define for $v \gg 0$, $\Delta':= \Delta + \frac{1}{p^v -1}(D + \Delta)$. Then if we replace $\Delta$ by $\Delta'$, the assumptions of Notation \ref{notation:semi_positivity} are still satisfied for every $v \gg 0$, even $p \nmid \ind(K_X + \Delta')$  by Lemma \ref{lem:index_divisibility_removal}. Furthermore, since $(X_0, \Delta_0)$ was strongly $F$-regular and  $\Delta'_0:= \Delta'|_{X_0}$ differs from $\Delta_0$ in a small enough divisor, $(X_0,\Delta_0')$ is sharply $F$-pure for every $v \gg 0$. 
Also, the $f$-nefness and the semi-ampleness assumptions  hold for $K_{X/Y} + \Delta'$, since 
\begin{equation}
\label{eq:index_p_divisible_nef_and_big:proportional}
K_{X/Y} + \Delta' \sim_{\bQ} \frac{p^v}{p^v-1} (K_{X/Y} + \Delta) + \frac{1}{p^v-1} A .
\end{equation} 
Hence, Theorem \ref{thm:relative_canonical_nef} implies that $K_{X/Y} + \Delta'$ is nef. However then using  \eqref{eq:index_p_divisible_nef_and_big:proportional}, so is 
\begin{equation*}
\frac{p^v -1}{p^v} \left( \frac{p^v}{p^v-1} (K_{X/Y} + \Delta) + \frac{1}{p^v-1} A \right) = (K_{X/Y} + \Delta) + \frac{1}{p^v} A .
\end{equation*}
Since this holds for all $v \gg 0 $, $K_{X/Y} + \Delta$ is nef. 

% In the case of point \eqref{itm:index_p_divisible_nef_and_big:push_forward}, we use a similar argument. We choose a $\Delta'$ as above. The only difference is that we choose a special $A:=l(K_{X/Y} + \Delta ) + B$, where $l$ is some (big enough) multiple of the index of $K_{X/Y} + \Delta$ and $B$ is the pullback of an ample (enough) divisor from $Y$. As above, all assumptions of Notation \ref{notation:semi_positivity} are satisfied for $K_{X/Y} + \Delta'$, and furthermore since the restriction of $A$ to the fibers is proportional to the restrictions of $K_{X/Y} + \Delta$, by \eqref{eq:index_p_divisible_nef_and_big:proportional}, $K_{X/Y} + \Delta'$ is $f$-semi-ample. Hence Theorem \ref{thm:pushforward_nef} implies that 
% \begin{equation*}
% f_* \sO_X(m(K_{X/Y} + \Delta')) \cong  \frac{p^v}{p^v-1} (K_{X/Y} + \Delta) + \frac{1}{p^v-1} A . 
% \end{equation*}

\end{proof}

\begin{lem}
\label{lem:S_0_monotone}
Let $(X,\Delta)$ and $(X,\Delta')$ be two pairs with the same underlying spaces, such that both $K_X + \Delta$ and $K_X + \Delta'$ are $\bQ$-Cartier with indices not divisble by $p$. Assume furthermore that $\Delta \leq \Delta'$. Then for any line bundle $\sL$ on $X$,
\begin{equation*}
S^0(X, \sigma(X,\Delta) \otimes \sL) \supseteq S^0(X, \sigma(X,\Delta') \otimes \sL) .
\end{equation*}
\end{lem}

\begin{proof}
For any integer $e>0$ for which both $(p^e-1)(K_X +\Delta)$ and $(p^e-1)(K_X +\Delta')$ are Cartier, let $\phi_{e,\Delta}$ (resp. $\phi_{e,\Delta'}$) denote the morphism $H^0(X,F^e_* \sL_{e,\Delta} \otimes \sL ) \to H^0(X, \sL)$ (resp. $H^0(X,F^e_* \sL_{e,\Delta'} \otimes \sL ) \to H^0(X, \sL)$) of Definition \ref{defn:S_0_F_pure_ideal}. It is enough to prove that for every $e>0$ as above, $\im \phi_{e,\Delta} \supseteq \im \phi_{e, \Delta'}$. However, that is immediate from the following factorization implied by  $(p^e-1)(K_X +\Delta) \leq (p^e-1)(K_X +\Delta')$.
\begin{equation*}
\xymatrix{
 \sL_{e, \Delta'} \cong \sO_X((1-p^e)(K_X + \Delta')) \ar[r] \ar@/^2pc/[rr]^{\phi_{e,\Delta'}} & \sL_{e, \Delta} \cong \sO_X((1-p^e)(K_X + \Delta)) \ar[r]_>>{\phi_{e,\Delta}} & \sO_X
}
\end{equation*}

\end{proof}

\begin{thm}
\label{thm:index_p_divisible_pushforward}
In the situation of Notation \ref{notation:semi_positivity} by possibly allowing $p|\ind(K_X + \Delta)$, assume also that $(X_0, \Delta_0)$ is strongly $F$-regular and $K_{X/Y}+\Delta$ is $f$-ample. Then $f_* \sO_X(mr(K_{X/Y} + \Delta))$ is a nef vector bundle for every $m \gg 0$.
\end{thm}

\begin{proof}
First, note that by Theorem \ref{thm:index_p_divisible_nef}, $K_{X/Y} + \Delta$ is nef.  Let  $A$ be the pullback of any ample Cartier divisor $B$ from $Y$. Then, by the $f$-ampleness of $K_{X/Y} + \Delta$, $r(K_{X/Y} + \Delta) + A $ is ample. Therefore, there is a $d > 0$, such that $dr(K_{X/Y} + \Delta) + dA + K_{X/Y}$ is linearly equivalent to an effective integer divisor $D$, which avoids the codimension 0 and the singular codimension 1 points of $X_0$ as well as the singular codimension 1 points of $X$. Let $n = p^v -1$ for arbitrary $v \gg 0$ (which notation will be used throughout the proof)  and define $\Delta_n:=\Delta + \frac{1}{n}(D + \Delta)$. By Lemma \ref{lem:index_divisibility_removal}, $(X,\Delta_n)$ satisfy the assumptions of Notation \ref{notation:semi_positivity} for $v \gg 0$, even the divisibility condition on the index. Furthermore the same holds for $(X^{(n)}, \Delta_n^{(n)})$ by Lemma \ref{lem:assumptions_for_product}. 

Fix now a $n'=p^{v'}-1 \gg 0$. Then, since $\Delta_{n'} -\Delta$ is a small effective divisor, $(X_0,(\Delta_{n'})_0)$ is sharply $F$-pure. In particular, by Theorem \ref{thm:relative_Fujita} and Corollary \ref{cor:Fujita_type_S_0_equals_H_0} we may choose an $M>0$ such that for all $i> 0$ and $m \geq M$,
\begin{equation}
\label{eq:index_p_divisible_pushforward:assumption_relative}
R^i f_* (\sO_X(mr(K_{X/Y}  + \Delta)))=0 \textrm{ and}
\end{equation}
\begin{equation}
\label{eq:index_p_divisible_pushforward:assumption_central_fiber}
  H^0(X_0, \sO_{X_0}(mr(K_{X/Y}  + \Delta) )) = S^0 ( X_0, \sigma ( X_0, (\Delta_{n'})_0 ) \otimes \sO_{X_0}(mr(K_{X/Y}  + \Delta) ) ) .
\end{equation}
Define then for any $m \geq \max\{2,M\}$ and $v \geq v'$ (still keeping the notation $n=p^v -1$), $N:= mr \left(K_{X^{(n)}/Y} + \Delta^{(n)} \right) + \frac{d}{n}A^{(n)}$. Then the assumptions of Proposition \ref{prop:semi_positive} hold for $\left(X^{(n)}, \Delta_n^{(n)} \right)$ and $N$, because if $\sN:= \sO_{X^{(n)}}(N)$:
\begin{itemize}
\item $N$ is Cartier, since $\frac{d}{n} A^{(n)} = \left( f^{(n)} \right)^*( dB)$.
\item  $R^i f_*^{(n)} \sN=0$ for all $i>0$ by \eqref{eq:index_p_divisible_pushforward:assumption_relative}, Lemmas \ref{lem:cohomology_and_base_change_pushforward_zero}, \ref{lem:cohomology_and_base_change_fibers_zero} and the K\"unneth formula. We also use here that $\sO_{X^{(n)}}(mr(K_{X^{(n)}/Y} + \Delta^{(n)})) \cong  \sO_X(mr(K_{X/Y} + \Delta))^{(n)}$.
\item $N - \left( K_{X^{(n)}/Y} + \Delta_n^{(n)} \right)$ is an $f^{(n)}$-ample and nef $\bQ$-divisor, because of the following computation and the nefness of $K_{X/Y} + \Delta$ granted by Theorem \ref{thm:index_p_divisible_nef}. Note that we  also use  that $\frac{n+1 + dr}{n} < 2$ because $v \geq v'$ and $v' \gg 0$.
\begin{align*}
\hspace{30pt} \mathrlap{N - \left(K_{X^{(n)}/Y} + \Delta_n^{(n)} \right) }
\\  & =
mr \left(K_{X^{(n)}/Y} + \Delta^{(n)} \right) + \frac{d}{n}A^{(n)} - \left( K_{X^{(n)}/Y} +  \Delta^{(n)} + \frac{1}{n} \left( D^{(n)} + \Delta^{(n)} \right) \right) 
\\ & \sim_{\bQ}
mr \left(K_{X^{(n)}/Y} + \Delta^{(n)}\right) + \frac{d}{n}A^{(n)} 
\\ & \hspace{100pt}- \left( K_{X^{(n)}/Y} +  \frac{n+1}{n}\Delta^{(n)} + \frac{1}{n}( dr(K_{X/Y} + \Delta) + dA + K_{X/Y})^{(n)}  \right)
\\ & =
mr \left(K_{X^{(n)}/Y} + \Delta^{(n)} \right)  - \frac{n+1}{n} \left( K_{X^{(n)}/Y} +  \Delta^{(n)} \right)  - \frac{dr}{n} \left(K_{X^{(n)}/Y} + \Delta^{(n)} \right),
\\ & =
\left( mr - \frac{n+1 + dr}{n}   \right) \left(K_{X^{(n)}/Y} + \Delta^{(n)} \right),
\end{align*}
\item $H^0(X_0^{(n)},\sN)= S^0 \left( X_0^{(n)}, \sigma \left(X_0^{(n)}, (\Delta_n^{(n)} )_0 \right) \otimes \sN|_{X_0^{(n)}} \right)$ by \eqref{eq:index_p_divisible_pushforward:assumption_central_fiber}, Lemma \ref{lem:S_0_monotone} and Lemma \ref{lem:product}.
\end{itemize}
Therefore, applying Proposition \ref{prop:semi_positive} yields that the following vector bundle is nef for every $v \gg 0$.
\begin{multline*}
f^{(n)}_* \sN
\cong   
f^{(n)}_* \sO_{X^{(n)}} \left( mr(K_{X^{(n)}/Y} + \Delta^{(n)}) + \frac{d}{n}A^{(n)} \right)
\\ \cong   
f^{(n)}_* \sO_{X^{(n)}} ( mr(K_{X^{(n)}/Y} + \Delta^{(n)})) \otimes \sO_Y(dB)
\cong   
\underbrace{\left( \bigotimes_{i=1}^n f_* \sO_{X} ( mr(K_{X/Y} + \Delta)) \right) \otimes \sO_Y(dB)}_{\textrm{by Lemma \ref{lem:cohomology_and_base_change_product}}}
\end{multline*}
Since this holds for every $n = p^v -1 \gg 0$, $f_* \sO_{X} ( mr(K_{X/Y} + \Delta)) $ is nef.
\end{proof}

\section{Applications}
\label{sec:applications}

\subsection{Projectivity of proper moduli spaces}
\label{sec:projectivity}

Recently there have been great advances in constructing  moduli spaces of varieties (or pairs) of (log-)general type in characteristic zero, c.f., \cite{Kollar_Moduli_of_varieties_of_general_type}. However, the method is not new. It has been worked out in \cite{Kollar_Projectivity_of_complete_moduli} and is  as follows. First, one defines a subfunctor of the functor of all families of (log-)canonically polarized varieties. Second, one proves nice properties of this functor: openness, separatedness, properness, boundedness and tame automorphisms. Then it follows that the chosen functor admits a coarse moduli space, which is a proper algebraic space. Third, by exhibiting a semi-positive line bundle on a finite cover of the functor that descends to the coarse moduli space, one proves that that the coarse moduli space is a projective scheme. Hence, Theorem \ref{thm:pushforward_nef_intro}  implies that in positive characteristics if a subfunctor as above of families of sharply $F$-pure varieties satisfies the 
first two steps, then the third step is satisfied as well. That is, it admits a coarse moduli space.

\begin{cor}
\label{cor:projectivity_of_moduli}
Let $\sF$ be a subfunctor of 
\begin{equation*}
Y \mapsto  \left. \left\{
\left.\raisebox{25pt}{   \xymatrix{ X \ar[d]_f \\ Y } } \right|
\parbox{280pt}{ \small
 $f : X \to Y$ is a flat, relatively $S_2$ and $G_1$, equidimensional, projective morphism with sharply $F$-pure fibers, such that
 there is a $p \nmid r>0$, for which   $\omega_{X/Y}^{[r]}$ is an $f$-ample line bundle, and $\Aut(X_y)$ is finite for all $y \in Y$  
}
\right\} \right/ \textrm{\small $\cong$ over $Y$}.
\end{equation*}
If $\sF$ admits
\begin{enumerate}
\item \label{assumption:projectivity_of_moduli:coarse_moduli} a coarse moduli space $\pi: \sF \to V$, which is a proper algebraic space and 
\item \label{assumption:projectivity_of_moduli:finite_cover} a morphism $\rho : Z \to \sF$ from a scheme, such that $\pi \circ \rho$ is finite,
%and for the family $g : W \to Z$ associated to $\rho$, $\det \left( g_* \omega_{W/Z}^{[mr]} \right)$ descends to $V$ for every high and divisible enough $m$,
\end{enumerate}
then $V$ is a projective scheme. 
\end{cor}

\begin{rem}
As mentioned in the introduction of this section, by \cite[Theorem 2.2]{Kollar_Projectivity_of_complete_moduli}, the assumptions of Corollary \ref{cor:projectivity_of_moduli} is satisfied if $\sF$  belongs to an open class with tame automorphisms and  is separated, bounded and complete. 
\end{rem}

\begin{proof}[Proof of Corollary \ref{cor:projectivity_of_moduli}]
To prove that $V$ is a projective scheme, one has to exhibit an ample line bundle on it. Let this be in our situation the descent of $\det \left( g_* \omega_{W/Z}^{[mr]} \right)$ to $V$ for some high and divisible enough $m$. Note that for divisible enough $m$, $\det \left( g_* \omega_{W/Z}^{[mr]} \right)$ descends indeed by the finiteness of the automorphism groups of the fibers, c.f., \cite[2.5]{Kollar_Projectivity_of_complete_moduli}. To prove that it is ample, it is enough to show that its pullback via $\pi \circ \rho$ is ample.  Therefore we are supposed to prove that  $\det \left( g_* \omega_{W/Z}^{[mr]} \right)$ is ample for some $m$ big and divisible enough.   However, since the isomorphism equivalence classes of the fibers of $g$ are exactly the fibers $\pi \circ \rho$, this ampleness is shown in Corollary \ref{cor:ample}.
\end{proof}

\subsection{Characteristic zero implications}

Fix an algebraically closed field $k'$ of characteristic zero throughout this section.
The following is a major conjecture in the theory of $F$-singularities (c.f., \cite{Mustata_Ordinary_varieties_and} \cite{Mustata_Srinivas_Ordinary_varieties_and}, \cite[Conjecture 1 and Corollary 4.5]{Miller_Schwede_Semi_log_canonical}). 

\begin{conj}
\label{conj:semi_log_canonical_reduction}
Given a pair $(X,\Delta)$ with semi-log canonical singularities over $k'$, consider a model $(X', \Delta')$ of it over a $\bZ$-algebra $A \subseteq k'$ of finite type.  Then there is a dense set of closed points $S \subseteq \Spec A$, such that $(X_s',\Delta_s')$ is sharply $F$-pure for all $s \in S$. 
\end{conj}

\begin{rem}
If in the above conjecture semi-log canonical is replaced by Kawmata log terminal and sharply $F$-pure by strongly $F$-regular, then then the statement is known \cite{Takagi_A_characteristic_p_analogue_of_plt_sintularities}.
\end{rem}

Hence, the results of the paper has the following consequences in characteristic zero. We emphasize this is a completely new algebraic method of obtaining such positivity results in characteristic zero.

\begin{cor}
\label{cor:characteristic_zero}
Let  $(X, \Delta)$ be a  pair over $k'$ with  $\bQ$-Cartier $K_X + \Delta$ and  $f : X \to Y$  a flat, projective morphism  to a smooth projective curve. Further suppose that there is a $y_0 \in Y$, such that $\Delta$ avoids all codimension 0 and the singular codimension 1 points of $X_{y_0}$, and either 
\begin{enumerate}
 \item $(X_{y_0} , \Delta_{y_0})$ is   Kawamata log terminal, or
 \item $(X_{y_0} , \Delta_{y_0})$ is semi-log-canonical and for every model over a $\bZ$-algebra $A$ of finite type, it satisfies the statement of Conjecture \ref{conj:semi_log_canonical_reduction}. 
\end{enumerate}
Assume also that $K_{X/Y} + \Delta$ is $f$-ample (resp. $f$-semi-ample). 
Then   for $m \gg 0$ and divisible enough, $f_* \sO_X(m(K_{X/Y} + \Delta))$ is a nef vector bundle (resp. $K_{X/Y} + \Delta$ is nef).
\end{cor}

\begin{proof}
Consider a model $f' :(X', \Delta') \to Y'$ of $f :(X, \Delta) \to Y$ over a $\bZ$-algebra $A \subseteq k'$ of finite type. By normalizing and then further localizing $A$, we may assume that
\begin{enumerate}
\item $\Spec A$ is Gorenstein,
\item $f'$ is flat,
\item $\Delta'$ avoids the codimension 0 and the singular codimension 1 points of $X_s'$,
\item $(X', \Delta')_s$ is a pair, i.e., $X$ is $S_2$ and $G_1$, for all $s \in \Spec A$,
\item $K_{X'} + \Delta'$ is $\bQ$-Cartier (note that if $r=\ind (K_{X'} + \Delta')$, then $r(K_{X'} + \Delta')|_{X_s} = r(K_{X_s'} + \Delta_s')$ in codimension one and then everywhere, therefore $\ind(K_{X_s'} + \Delta_s') | \ind (K_{X'} + \Delta')$),
\item $\mathrm{char}(s) \nmid \ind(K_{X'}+ \Delta')$ for every $s \in S$ (and then by the above considerations, $\mathrm{char}(s) \nmid \ind(K_{X'_s}+ \Delta'_s)$),
\item $Y'_s$ is smooth for every $s \in S$,
\item $\Delta$ avoids all codimension 0 and the singular codimension 1 points of $X_{(y_0,s)}$ for all $s \in \Spec A$,
\item $K_{X'/Y'} + \Delta'$ is $f'$-ample (resp. $f'$-semi-ample) and 
\item in the $f'$-ample case, we may also assume that $m$ is chosen big and divisible enough such that $f_*' \sO_X(m (K_{X'/Y'} + \Delta'))|_{Y_s} \cong (f_s')_* \sO_X(m (K_{X'_s/Y'_s} + \Delta'_s))$ and the same for $s$ replaced by $\bar{s}$ (if $s$ was given by a morphism $A \to k''$, then $\bar{s}$ denotes a morphism given by $A \to \bar{k''}$, where $\bar{k''}$ is any algebraic closure of $k''$).
\end{enumerate}
Note also that by the assumptions there is a dense set $S \subseteq \Spec A$ of closed points for which $\left( X_{(y_0,s)}', \Delta_{(y_0,s)}' \right)$ is $F$-pure. In particular then for every $s \in S$, $(X_s', \Delta_s')$ satisfy the assumptions of Notation \ref{notation:semi_positivity} and Theorem \ref{thm:pushforward_nef} (resp. Theorem \ref{thm:relative_canonical_nef}) except that the base field is not algebraically closed. However, the above assumptions   are stable under passing to the algebraic completion of the base-field. Therefore, $(X_{\bar{s}}', \Delta_{\bar{s}}')$ satisfies the assumptions of Notation \ref{notation:semi_positivity} and Theorem \ref{thm:pushforward_nef} (resp. Theorem \ref{thm:relative_canonical_nef}) for all $s \in S$ including the algebraic closedness of the base field.  In particular, by Theorem \ref{thm:pushforward_nef} (resp. Theorem \ref{thm:relative_canonical_nef}), $(f_*' \sO_X(m(K_{X'/Y'} + \Delta')))_{\bar{s}}$ is a nef vector bundle (resp. $(K_{X/Y} + \Delta)_{\bar{
s}}$ is nef) for every $s \in S$. However, then so is $(f_*' \sO_X(m(K_{X'/Y'} + \Delta')))_{s}$ (resp. $(K_{X/Y} + \Delta)_{s}$). By \cite[Proposition 1.4.13]{Lazarsfeld_Positivity_in_algebraic_geometry_I}, nefness at a point implies nefness at all its generalizations. Hence   $f_* \sO_X(m (K_{X/Y} + \Delta))$ (resp. $ K_{X/Y} + \Delta$) is nef.
\end{proof}

\subsection{Subadditivity of Kodaira-dimension}

Subadditivity of Kodaira dimension was one of the major applications in characteristic zero of the semipositivity of $f_* \omega_{X/Y}^{m}$ (e.g., \cite{Viehweg_Weak_positivity}, \cite{Kollar_Subadditivity_of_the_Kodaira_dimension}). We present here a similar result in positive characteristic. However, we would to draw the reader's attention that in positive characteristic,  there is already some ambiguity to the notion of Kodaira dimension. See Question \ref{qtn:Kodaira_dimension} for explanation. Hence we have to phrase the statement slightly differently, involving the bigness of (log-)canonical divisors instead of the (log-)Kodaira dimension.

\begin{cor}
\label{cor:subadditivity_general_type_base}
In the situation of Notation \ref{notation:results}, if furthermore $Y$ is an $S_2, G_1$, equidimensional projective variety with $K_Y$ $\bQ$-Cartier and big,  $K_{X/Y} + \Delta$ is $f$-semi-ample and  $K_F + \Delta|_F$ is big for the generic fiber $F$, then $K_X + \Delta$ is  big.
\end{cor}

\begin{proof}
Since, $K_Y$ is big, there is a $m>0$, such that $m K_Y = A + E$ for integer very ample and effective divisors $A$ and $E$. It is enough to prove that $f^*A + m (K_{X/Y} + \Delta)$ is big. By Theorem \ref{thm:relative_canonical_nef_2_intro}, $K_{X/Y} + \Delta$ is nef. So, since $f^*A + m (K_{X/Y} + \Delta)$ is nef, it is enough to show that $(f^* A + m (K_{X/Y}+ \Delta))^{\dim X} > 0$. However then the following computation concludes our proof.
\begin{multline*}
 (f^* A + m(K_{X/Y}+\Delta))^{\dim X} 
\geq 
\underbrace{f^* A^{\dim Y} \cdot m(K_{X/Y} + \Delta)^{\dim F}}_{\textrm{both $f^*A$ and $K_{X/Y} + \Delta$ are nef}}
\\ =
A^{\dim Y} \cdot (m(K_F + \Delta|_F))^{\dim F} 
>
\underbrace{0}_{K_F + \Delta|_F\textrm{ is big}},
\end{multline*}
\end{proof}

\section{Questions}
\label{sec:questions}

Here we list questions that are left open by the article and we feel are important. We feel that the next one is the most important.

\begin{question}
Given a projective, sharply $F$-pure pair $(X, \Delta)$ with $K_X + \Delta$ $\bQ$-Cartier and ample, is $\Aut(X,\Delta)$ finite? More generally one might ask the same question but for semi-log canonical pairs in positive characteristics instead of sharply $F$-pure pairs.
\end{question}

The next question is motivated by the absence of resolution of singularities in positive characteristics. Recall, that the Kodaira dimension of a variety $X$ in characteristic zero is defined as the Kodaira dimension of $K_{X'}$ for a projective smooth birational model $X'$ of $X$.

\begin{question}
\label{qtn:Kodaira_dimension}
Is there a birational invariant in positive characteristics, which specializes to Kodaira dimension in the particular case when there is a smooth birational model? In particular this question would be solved if we had resolution of singularities in positive characteristics.
\end{question}

The following few questions concern sharpness of theorem \ref{thm:pushforward_nef_intro}.

\begin{question}
Can one drop the  index not divisible by $p$ assumption in Theorem \ref{thm:pushforward_nef_intro}?
\end{question}

\begin{question}
Is there a family $f : X \to Y$ of semi-log canonical but not sharply $F$-pure schemes over a projective smooth curve with $K_X$ is $\bQ$-Cartier $f$-ample, such that $f_* \omega_{X/Y}^{[m]}$ is not nef for every $m \gg 0$?
\end{question}

\begin{question}
Can one give an effective bound on $m$ for which Theorem \ref{thm:pushforward_nef_intro} holds? Is it possibly true for $m \geq 2$?
\end{question}

\bibliographystyle{skalpha}
\bibliography{includeNice}
 
\end{document}